\documentclass{aims}
\usepackage{amsmath}
  \usepackage{paralist}
  \usepackage{graphics} 
  \usepackage{epsfig} 
\usepackage{graphicx}  \usepackage{epstopdf}
 \usepackage[colorlinks=true]{hyperref}
\usepackage[all]{xy} 
\usepackage{amssymb} 
\hypersetup{urlcolor=blue, citecolor=red}
\usepackage{multicol}
\usepackage[active]{srcltx}
\usepackage{amsthm}
\usepackage{amsmath}
\usepackage{wasysym}
\def\currentvolume{X}
 \def\currentissue{X}
  \def\currentyear{200X}
   \def\currentmonth{XX}
    \def\ppages{X--XX}
     \def\DOI{10.3934/xx.xx.xx.xx}

\newtheorem{theorem}{Theorem}[section]
\newtheorem{corollary}{Corollary}

\newtheorem{lemma}[theorem]{Lemma}

\newtheorem{proposition}{Proposition}

\theoremstyle{definition}
\newtheorem{definition}[theorem]{Definition}
\newtheorem{remark}{Remark}

\newtheorem{example}{Example}

\newcommand{\llbracket}{\lbrack\! \lbrack}
\newcommand{\rrbracket}{\rbrack\! \rbrack}

\newcommand{\ppig}{\mathcal{P}^{\pi}G} 
\newcommand{\palg}{\mathcal{P}^{\alpha}G}    
\newcommand{\palag}{\mathcal{P}^{\alpha}(AG)}
\newcommand{\palgt}{\mathcal{P}^{\alpha'}G'}
\newcommand{\alpalg}{A\left(\mathcal{P}^{\alpha}G\right)}
\newcommand{\algpalg}{A_g\left(\mathcal{P}^{\alpha}G\right)}
\newcommand{\pgalag}{\mathcal{P}_{g}^{\alpha}(AG)}

\def\lcf{\lbrack\! \lbrack}
\def\rcf{\rbrack\! \rbrack}
\newcommand{\lp}{\left(}
\newcommand{\rp}{\right)}
\newcommand{\lc}{\left\{}
\newcommand{\rc}{\right\}}

\newcommand{\bra}{\langle}
\newcommand{\ket}{\rangle}
\newcommand{\R}{\mathbb{R}}      
\newcommand{\F}{\mathbb{F}}

\newcommand{\Flder}{\rightarrow}
\newcommand{\hooklongrightarrow}{\lhook\joinrel\longrightarrow}

\usepackage{amssymb}

\newcommand{\proa}{A^*G \mbox{$\;$}_{\tau^*} \kern-3pt\times_\alpha
G \mbox{$\;$}_\beta \kern-3pt\times_{\tau^*} A^*G}

\newcommand{\lvec}[1]{\overleftarrow{#1}}
\newcommand{\rvec}[1]{\overrightarrow{#1}}
\newcommand{\alg}{\mathfrak{so}(3)}

\newcommand{\Ad}{\mbox{Ad}}
\newcommand{\ca}{\mbox{cay}}
\newcommand{\ad}{\mbox{ad}}
\newcommand{\Om}{\Omega}

\newcommand{\al}{\mathfrak{g}}

\title[Second-order variational problems on Lie groupoids] 
      {Second-order variational problems on Lie groupoids and optimal control applications}

\author[Leonardo Colombo and David Mart\'in de Diego]{}


 \email{ljcolomb@umich.edu}
 \email{david.martin@icmat.es}

\begin{document}
\maketitle

\centerline{\scshape Leonardo Colombo}
\medskip
{\footnotesize
 \centerline{Department of Mathematics}
 \centerline{University of Michigan}
   \centerline{530 Church Street, 3828 East Hall}
   \centerline{ Ann Arbor, Michigan, 48109, USA}
} 

\medskip

\centerline{\scshape David Mart\'in de Diego}
\medskip
{\footnotesize
 \centerline{ Instituto de Ciencias Matem\'aticas (CSIC-UAM-UCM-UC3M)}
   \centerline{Calle Nicol\'as Cabrera 15, Campus UAM, Cantoblanco}
   \centerline{Madrid, 28049, Spain}
}

\bigskip

 \centerline{(Communicated by the associate editor name)}

\begin{abstract}
In this paper we study, from a variational and geometrical point of view, second-order variational problems on Lie groupoids and the construction of variational integrators for optimal control problems.  First, we develop variational techniques for
second-order variational problems on Lie groupoids and their applications
 to the construction of variational integrators
for optimal control problems of mechanical systems. Next, we
show how Lagrangian submanifolds of a symplectic groupoid
gives intrinsically the discrete dynamics for second-order systems, both unconstrained and constrained, and we study the geometric properties of the implicit flow which defines the dynamics in the Lagrangian submanifold. We also study the theory of
reduction by symmetries and the corresponding Noether theorem.  

\end{abstract}

\section{Introduction}

The topic of discrete Lagrangian mechanics concerns the study of
certain discrete dynamical systems on manifolds. As the name
suggests, these discrete systems exhibit many geometric features
which are analogous to those in continuous Lagrangian mechanics. In
particular, the discrete dynamics  are derived  from  variational
principles, have symplectic or Poisson flow maps, conserve momentum
maps associated to Noether-type symmetries, and admit a theory of
reduction. While discrete Lagrangian systems are quite
mathematically interesting, in their own right, they also have
important applications in the design of structure-preserving numerical methods for
many dynamical systems in  mechanics and optimal control
theory.

Numerical methods which are constructed in this way are called
variational integrators. This approach to discretizing Lagrangian
systems was put forward in papers by Bobenko and Suris \cite{BoSu},
Moser and Veselov \cite{moserveselov}, and others in the early
1990s, and the general theory was developed over the subsequent
decade (see Marsden and West \cite{MaWest} for a comprehensive
overview).

A. Weinstein \cite{We} observed that these systems could be
understood as a special case of a more general theory, describing
discrete Lagrangian mechanics on arbitrary Lie groupoids where the Lagrangian function is defined on a Lie groupoid.
A Lie groupoid $G$ is a natural generalization of the concept of a
Lie group, where now not all elements are composable. The product
$gh$ of two elements is only defined on the set of composable
pairs $G_{2}=\lc(g,h)\in G\times G\,|\,\beta(g)=\alpha(h)\rc$, where
$\alpha:G\Flder Q$  and $\beta:G\Flder Q$ are the source and target
maps over a base manifold $Q$. This concept was introduced in
differential geometry by Ereshmann in the 1950's. The infinitesimal
version of a Lie groupoid $G$ is the Lie algebroid $\tau_{AG}:AG\Flder Q$,
which is the restriction of the vertical bundle of $\alpha$ to the
submanifold of the identities. This setting is general enough to include the discrete counterparts of several types of fundamental equations in Mechanics as for instance, Euler-Lagrange equations for Lagrangians defined on tangent bundles \cite{MaWest}, Euler-Poincar\'e equations for Lagrangians defined on Lie algebras \cite{Mars3}, \cite{MaPeSh}, Lagrange-Poincar\'e equations for Lagrangians defined on Atiyah bundles, etc. Such discrete counterparts are  obtained discretizing the continuous Lagrangian to the corresponding Lie groupoid and then applying a suitable discrete variational principle.

A complete description of the discrete Lagrangian and Hamiltonian
mechanics on Lie groupoids was given in the work by Marrero,
Mart\'in de Diego and Mart\'inez \cite{MMM} (see also \cite{MMM2}). Following the program proposed by A. Weinstein \cite{We}, in this work, we
generalize the theory of discrete second-order Lagrangian mechanics
and variational integrators for a second-order discrete Lagrangian
$L_d:G_2\to\R$ in two main directions. First, we develop variational
principles for second-order variational problems on Lie groupoids
and we show how to apply this theory to the construction of
variational integrators for optimal control problems of mechanical
systems. Secondly, we show that Lagrangian submanifolds of a
symplectic groupoid (cotangent groupoid) give rise to discrete
dynamical second-order systems. We also develop a
reduction by  symmetries, and study the relationship between
the dynamics and variational principles for second-order
variational problems. Finally we study discrete second-order constrained Lagrangian mechanics on Lie groupoids. The main application of this theory will be  systems  subjected  to constraints
and underactuated control systems.
Our results in the second part of the paper are based on the papers   \cite{MMM} and \cite{MMS} but in second-order theories. 

There are variational principles which involves higher-order
derivatives \cite{BlochCrouch2}, \cite{BHM}, \cite{MR1333770}, \cite{GHMRV10}, \cite{GHMRV12}, \cite{GHR12}, \cite{MR2580471}, \cite{Eduardoho}, \cite{DavidM}, \cite{MR1036158}, \cite{PrietoRoman-Roy1} since from it one can obtain the equations
of motion for Lagrangians where the configuration space is a
higher-order tangent bundle. The study of higher-order variational systems has regularly attracted a lot of attention
from the applied and theoretical points of view (see \cite{LR1}), but
recently, higher-order variational problems have been studied for
their important applications in optimal control, aeronautics, robotics,
computer-aided design, air traffic control, trajectory planning and computational anatomy. 

Since the main applications of higher-order variational problems are second-order problems we will focus our attention in the second-order case along the work, leaving the extension to higher-order case as a straightforward development. 

The organization of the paper is as follows. In Section 2 we recall some constructions and results on discrete Mechanics on Lie groupoids which will be used in the next sections.  Section 3 is devoted to study variational principles for second-order discrete mechanical systems on Lie groupoids, their extension to the constrained case and the application to the theory of optimal control of mechanical systems and construction of variational integrators. In Section 4
we show how  Lagrangian submanifolds of an appropriate 
symplectic groupoid (cotangent groupoid) give rise to discrete
dynamical second-order systems. From such Lagrangian submanifold we obtain the discrete second-order Euler-Lagrange equations on Lie groupoids and such equations correspond with the ones obtained from the variational point of view. We also develop a
reduction by  symmetries, and we study the relationship between
the different dynamics and variational principles for these second-order
variational problems. Finally we study discrete constrained second-order Lagrangian mechanics. This
allows for systems with arbitrary constraints.

Throughout the paper, we have occasion to draw on certain technical constructions using the theory of Lie algebroids and retraction maps. We provide
some supplementary details and discussion of these and an overview on discrete mechanics and higher-order tangent bundles in some Appendices at the end of this paper.

\section{Groupoids and discrete mechanics}\label{LG}

This section review some results about Lie groupoids and discrete mechanics on Lie groupoids based on \cite{MMM} and \cite{MMM2}.

\subsection{Generalities about Lie groupoids} A \textit{groupoid} is a small category in
which every morphism is an isomorphism (i.e. all morphism is
invertible). That is,
\begin{definition}
 A {\it groupoid} $G$ over a
set $Q,$ denoted $ G \rightrightarrows Q
  $, consists of a set of objects $Q$, a set of morphisms $G$, and the
  following structural maps:
\begin{itemize}
\item a {\it source map} $ \alpha \colon G \rightarrow Q $ and a {\it target map} $
  \beta \colon G \rightarrow Q $. Thus an element $g\in G$ is thought as an arrow from $x=\alpha(g)$ to $y=\beta(g)$ in $Q$. 
  $$
\xymatrix{*=0{\stackrel{\bullet}{\mbox{\tiny
 $x=\alpha(g)$}}}{\ar@/^1pc/@<1ex>[rrr]_g}&&&*=0{\stackrel{\bullet}{\mbox{\tiny
$y=\beta(g)$}}}}
$$
\item an associative {\it multiplication map} $ m \colon G _2 \rightarrow G
  $, $ m(g,h)= gh $, with $\alpha(g)=\alpha(gh)$ and
  $\beta(h)=\beta(gh)$ where
\begin{equation*}
  G _2:= G \mathrel{_\beta \times _\alpha} G = \left\{ \left( g , h
    \right) \in G \times G \;\middle\vert\; \beta (g) = \alpha (h)
  \right\}
\end{equation*}
is called set of {\it composable pairs} defined by $\alpha$ and $\beta$. $gh$ is thought as the composite
arrow from $x$ to $z$ if $g$ is an arrow from $x=\alpha(g)$ to $y=\beta(g)=\alpha(h)$ and $h$ is an arrow from $y$ to $\beta(h)=z$.
$$\xymatrix{*=0{\stackrel{\bullet}{\mbox{\tiny
 $x=\alpha(g)=\alpha(gh)$}}}{\ar@/^2pc/@<2ex>[rrrrrr]_{gh}}{\ar@/^1pc/@<2ex>[rrr]_g}&&&*=0{\stackrel{\bullet}{\mbox{\tiny
 $y=\beta(g)=\alpha(h)$}}}{\ar@/^1pc/@<2ex>[rrr]_h}&&&*=0{\stackrel{\bullet}{\mbox{\tiny
 $z=\beta(h)=\beta(gh)$}}}}$$
\item an {\it identity map}, $ \epsilon \colon Q \rightarrow G,$ a section of $\alpha$ and $\beta$, such
  that for all $ g \in G $,
\begin{equation*}
  \epsilon \left( \alpha (g) \right) g = g = g \epsilon \left( \beta
    (g) \right),
\end{equation*}
\item an {\it inversion map} $ i \colon G \rightarrow G $, mapping $g$ into
  $g ^{-1}$, such that for all $ g \in G $,
\begin{equation*}
  g g ^{-1} = \epsilon \left( \alpha (g) \right) \hbox{ and } g ^{-1} g =
  \epsilon \left( \beta (g) \right) .
\end{equation*}

$$\xymatrix{*=0{\stackrel{\bullet}{\mbox{\tiny
 $x=\alpha(g)=\beta(g^{-1})$}}}{\ar@/^1pc/@<2ex>[rrr]_g}&&&*=0{\stackrel{\bullet}{\mbox{\tiny
 $y=\beta(g)=\alpha(g^{-1})$}}}{\ar@/^1pc/@<2ex>[lll]_{g^{-1}}}}$$

\end{itemize}
\end{definition}

\begin{remark}

Alternatively, a groupoid can be seen as a weak version of a group,
where the multiplication will be defined only for elements in
$G_2\subset G\times G$.\hfill$\diamond$
\end{remark}

We will focus on a particular class
of groupoids, the Lie groupoids which have a differential structure in addition to their algebraic structure.

\begin{definition}
  \label{def:liegroupoid}
  A {\it Lie groupoid} is a groupoid $ G \rightrightarrows Q $ where
\begin{enumerate}
\item $G$ and $Q$ are differentiable manifolds,
\item $ \alpha, \beta $ are submersions,
\item the multiplication map $ m,$ inversion $i$, and identity $\epsilon$, are differentiable.
\end{enumerate}
\end{definition}

\begin{remark}
  In Definition~\ref{def:liegroupoid}, $\alpha$ and $\beta$ must be
  submersions so that $ G _2 $ is a differentiable manifold. From the definition it follows that $m$ is a submersion,
  $\epsilon$ is an immersion, and $i$ is a diffeomorphism.\hfill $\diamond$
\end{remark}

If $q\in Q$, $\alpha^{-1}(q)$ and $\beta^{-1}(q)$ will be said the $\alpha$-fiber and $\beta$-fiber of $q$.




\begin{definition}\label{leftrighttranslationgroupoids}
  Given a groupoid $ G \rightrightarrows Q$ and $ g \in G
  $, define the {\it left translation} $ \ell _g \colon \alpha ^{-1}
  \left( \beta (g) \right) \rightarrow \alpha ^{-1} \left( \alpha (g)
  \right) $ and {\it right translation} $ r _g \colon \beta ^{-1}
  \left( \alpha (g) \right) \rightarrow \beta ^{-1} \left( \beta (g)
  \right) $ by $g$ to be
\begin{equation*}
  \ell _g (h) = g h \hbox{ and } r _g (h) = h g .
\end{equation*} Note that, $\ell^{-1}_g=\ell_{g^{-1}}$ and $r^{-1}_g=r_{g^{-1}}.$
\end{definition}

Denoting by $\mathfrak{X}(G)$ the set of vector fields on $G$ one may introduce the notion of a left (right)-invariant vector
field in a Lie groupoid, as in the case of Lie groups.

\begin{definition}
  Given a Lie groupoid $ G \rightrightarrows Q $, a vector field  $X\in \mathfrak{X} (G)$ is {\em left-invariant} (resp., {\em right-invariant}) if $X$ is
  $\alpha$-vertical (resp., $\beta$-vertical), that is, it is tangent to the fibers of $\alpha$ (resp., $\beta$), 
  $T\alpha(X)=0$ (resp., $T\beta(X)=0$) and
  $ \left( T _h \ell_g \right) \left( X (h) \right) = X\left( g h \right) $
  for all $ (g, h) \in G _2$  (resp.,
  $ \left( T _h r _g \right) \left( X (h) \right) = X\left( h g
  \right)$).
\end{definition}

\begin{definition}
A \emph{Lie algebroid} $A$ over a manifold $Q$ is a
real vector bundle $\tau_{A}:A\to Q$ together with a Lie bracket
$\llbracket\cdot,\cdot\rrbracket$ on $\Gamma(\tau_{A})$, the set of sections of $\tau_{A}:A\to Q$, and a bundle
map $\rho:A\to TQ$ such that $\llbracket X,fY\rrbracket=f\llbracket
X,Y\rrbracket+\rho(X)(f)Y$ for all $X,Y\in\Gamma(\tau_{A})$ and
$f\in C^{\infty}(Q)$.
\end{definition}
\begin{remark} If $(A,\llbracket\cdot,\cdot\rrbracket,\rho)$ is a Lie algebroid over
$Q$ then $\rho$ is a homomorphism between the Lie algebras
$(\Gamma(\tau_{A}),\llbracket\cdot,\cdot\rrbracket)$ and
$(\mathfrak{X}(Q),[\cdot,\cdot])$.\hfill$\diamond$ 
\end{remark}

In Lie groups, the infinitesimal version of a Lie group is a Lie
algebra, therefore we will see that the corresponding infinitesimal version of a
Lie groupoid is a Lie algebroid. Next, we define the
\textit{Lie algebroid associated with a Lie groupoid $G\rightrightarrows Q$.}

Given a Lie groupoid $ G \rightrightarrows Q $, consider the
vector bundle $$\tau_{AG}:AG\longrightarrow Q$$ whose fiber at a point $q\in Q$ is $A_qG =\ker T _{
\epsilon(q)}\alpha=V\alpha $, i.e., the tangent space to the $\alpha$-fiber
at the identity section, for $ q \in Q $.

It is easy to prove that there exists a bijection between the space of sections $\Gamma
\left(\tau_{AG}\right)$ and the set of left-invariant vector fields on $G$.  If $X$ is a section of $\tau_{AG}$, the corresponding left-invariant vector field on $G$ will be denoted by $\overleftarrow{ X } \in \mathfrak{X}(G)$ where
\begin{equation}
\label{left-invariant-vf}
\overleftarrow{ X } (g) = \left( T _{ \epsilon \left( \beta (g) \right) }
\ell _g \right) \left( X \left( \beta (g) \right) \right) .
\end{equation}

The Lie algebroid structure on $AG$ is given by the bracket $\lcf\cdot,\cdot\rcf$
and the anchor map $\rho$ defined as follows:
\begin{equation}\label{estructuraalgebroide}
 \overleftarrow{\llbracket{X , Y}\rrbracket} = \bigl[
  \overleftarrow{X} , \overleftarrow{Y} \bigr] , \qquad \rho
  (X) (q)  = \left( T _{ \epsilon (q) } \beta \right) \left( X (q)
  \right) ,
\end{equation}
for all $ X, Y \in \Gamma \left( \tau_{AG} \right) $ and $ q \in Q $.

\begin{definition}
Given a Lie groupoid $G \rightrightarrows Q$, the triple
$(AG,\lcf\cdot,\cdot\rcf,\rho)$ defined in \eqref{estructuraalgebroide} is called \textit{Lie algebroid
associated to} $G \rightrightarrows Q$.
\end{definition}

\begin{remark}
Alternatively one can also establish a bijection between the space of sections $\Gamma
  \left( \tau_{AG} \right) $ and the set of right-invariant vector fields on $
  Q$, by
\begin{equation}
\label{right-invariant-vf}
  \overrightarrow{ X } (g) = - \left( T _{ \epsilon \left( \alpha (g)
      \right) } \left( r _g \circ i \right) \right) \left( X \left(
      \alpha (g) \right) \right) ,
\end{equation}
which yields the Lie bracket relation
\begin{equation*}
  \overrightarrow{\llbracket{X , Y}\rrbracket} = -\bigl[
  \overrightarrow{ X }, \overrightarrow{ Y } \bigr], \qquad \bigl[\overrightarrow{X},\overleftarrow{Y}\bigr]=0.
\end{equation*}
Thus the  mapping $X$ into $\overleftarrow{ X } $ is a Lie algebra isomorphism,
and the mapping $X$ into $\overrightarrow{ X } $ is a Lie algebra
anti-isomorphism (see \cite{Weinstein-symplectic} and \cite{Mack} fore more details).\hfill$\diamond$
\end{remark}

\subsubsection{Examples of Lie groupoids:}\label{examples} We introduce some examples of Lie groupoids.

\vspace{0.3cm}

$\bullet$ \textit{The pair or banal groupoid:} Let $Q$ be a
differentiable manifold, and consider the product manifold
$G=Q\times Q$.  $G$ is a Lie groupoid over $Q$ where the source
and target maps $\alpha$ and $\beta$ are the projections onto the
first and second factors respectively. The identity is defined as
$\epsilon(q)=(q,q)$ for all $q\in Q,$ the multiplication
$m((q,s),(s,r))=(q,r)$ for $(q,s),(s,r)\in Q\times Q$ and the
inverse map $i(q,s)=(s,q).$

Note that, if $q$ is a point of $Q$, then
$V_{\epsilon(q)}\alpha\simeq T_{q}Q$. Hence the Lie algebroid of $G$ is isomorphic to the standard Lie algebroid $\tau_{TQ}:TQ\to Q$. In this sense, the Banal groupoid is considered as the discrete space for discretizations of Lagrangian functions $L:TQ\to\R$.

\vspace{0.3cm}

$\bullet$ \textit{Lie groups:} Let $G$ be a Lie group.
$G$ is a Lie groupoid over one point $Q=\{\mathfrak{e}\},$ the identity element
of $G.$ The structural maps of the
Lie groupoid $G$ are
$$\alpha(g)=\mathfrak{e},\quad \beta(g)=\mathfrak{e},\quad \epsilon(\mathfrak{e})=\mathfrak{e}, \quad i(g)=g^{-1} \hbox{ and } m(g,h)=gh, \hbox{ for }g,h\in G.$$
The Lie algebroid associated with $G$ is the Lie algebra
$\mathfrak{g}$ of $G$.

\vspace{0.3cm}

$\bullet$ \textit{Transformation or action Lie groupoid}. Let $H$ be a Lie group with identity $\tilde{\mathfrak{e}}$ and  $\varphi: Q \times H\to Q$ be a right action of $H$ on $Q$. The product manifold $G= Q\times H$  is a Lie groupoid over $Q$, with structural maps given by 
\begin{align*}
&\alpha(q,\tilde{h})= q,\quad \beta(q,\tilde{h})= \varphi(q,\tilde{h}),\quad \epsilon(q)=(q,\tilde{\mathfrak{e}}),\quad m((q,\tilde{h}),(\varphi(q,\tilde{h}),\tilde{h}'))=(q,\tilde{h}\tilde{h}'),\\
&i(q,\tilde{h})= (\varphi(q,\tilde{h}),\tilde{h}^{-1}) \hbox{ for } q\in Q \hbox{ and }h,h'\in H. 
\end{align*} The Lie groupoid $G$ is called action or transformation Lie groupoid and its associated Lie algebroid is the action algebroid $\hbox{pr}_1:M\times\tilde{\mathfrak{h}}\to M$ where $\tilde{\mathfrak{h}}$ is the Lie algebra of the Lie group $H$ (for more details, see \cite{Mack} and \cite{MMM}).

$\bullet$ \textit{The cotangent groupoid:} Let $G\rightrightarrows
Q$ be a Lie groupoid. If $A^\ast G$ is the dual bundle to $AG$ then
the cotangent bundle $T^\ast G$ is a Lie groupoid over $A^\ast G$.
The projections $\tilde\beta$ and $\tilde\alpha$, the partial
multiplication $\oplus _{T^\ast G}$, the identity section
$\tilde\epsilon$ and the inversion $\tilde{\i}$ are defined by the structural maps of $G$ as
follows,
\begin{equation}\label{eq:cotangent:groupoid}
\kern-15pt
\begin{array}{l} \tilde\beta(\mu _g)(X)=\mu _g
((T_{\epsilon (\beta(g))} \ell_g )(X)), \mbox{ for }\mu _g\in
T^\ast _gG \mbox{ and }X\in A_{\beta(g)}G, \\[5pt] \tilde\alpha (\nu
_h)(Y)=\nu _h ((T_{\epsilon (\alpha (h))} r_h) (Y-
(T _{\epsilon (\alpha (h))}(\epsilon \circ \beta)) (Y))),\\[4pt]\kern160pt
\mbox{ for }\nu _h\in T^\ast _hG \mbox{ and }Y\in A_{\alpha(h)}G ,
\\[5pt] (\mu _g\oplus _{T^\ast G}\nu _h)(T_{(g,h)}m (X_g,Y_h))=\mu
_g(X_g)+\nu _h (Y_h),\\[4pt]\kern160pt\mbox{ for }(X_g,Y_h)\in
T_{(g,h)}G_{2},
\\[5pt] \tilde\epsilon(\mu _x)(X_{\epsilon (x)})=\mu _x(X_{\epsilon
(x)}- (T_{\epsilon (x)}(\epsilon \circ \alpha )) (X_{\epsilon
(x)}))),\\[4pt] \kern160pt \mbox{ for }\mu _x\in A^\ast _xG \mbox{
and }X_{\epsilon (x)}\in T_{\epsilon (x)}G ,\\[5pt] \tilde{\i}
(\mu _g)(X_{g^{-1}})=-\mu _g((T_{g^{-1}} i )(X_{g^{-1}})), \mbox{
for }\mu _g\in T^\ast _gG\mbox{ and }X_{g^{-1} }\in T_{g^{-1}}G.
\end{array}
\end{equation}
Here $\alpha,\beta,m,i,\epsilon$ are the structural maps of $G\rightrightarrows Q$ (for more details, see \cite{Weinstein-symplectic} and \cite{IMMP}).

$\bullet$ \textit{Symplectic groupoids:} Finally, we introduce a
subclass of Lie groupoids with an additional structure, symplectic
groupoids. They are endowed with a symplectic manifold structure. 
A {\em symplectic groupoid} is a Lie groupoid $G\rightrightarrows Q$
with a symplectic form $\omega$ on $G$ such that the graph of the
composition law $m$ given by $$\hbox{graph}(m):=\{(g,h,r)\in G\times G\times G\mid
(g,h)\in G_2 \hbox{ and } r=m(g,h)\}$$ is a Lagrangian submanifold
of $G\times G\times G^{-}$ with the product symplectic form, where
the first two factors $G$ are endowed with the symplectic form
$\omega$ and the third factor $G^{-}$ with the symplectic
form $-\omega.$

Observe that if $G\rightrightarrows Q$ is a Lie groupoid then the cotangent
groupoid $T^{*}G\rightrightarrows A^{*}G$ is a symplectic groupoid
with the canonical symplectic $2$-form on $T^{*}G$, denoted $\omega_{G}$.
\subsection{Lie Groupoids and Discrete Mechanics}

Next, we give a review of some generalities on discrete mechanics on Lie groupoids based in \cite{MMM} and \cite{MMM2}.
\subsubsection{Discrete Euler-Lagrange equations} Let $G$ be  a Lie
groupoid over $Q$ with structural maps
\[
\alpha, \beta: G \longrightarrow Q, \; \; \epsilon: Q \longrightarrow G, \; \; i: G \longrightarrow G,
\; \; m: G_{2} \longrightarrow G.
\]
Denote by $\tau_{AG}:AG\to Q$ the Lie algebroid of $G$.

 A \emph{discrete Lagrangian} is a function $\mathbb{L}_d:{G}\to\R$. Fixed $g\in
G$, we define the set of admissible sequences with values in $G$:
\begin{equation}\label{sequences}
{\mathcal C}^N_{g}=\{(g_1, \ldots, g_N)\in G^N \mid (g_k,
g_{k+1})\in G_2 \hbox{ for } k=1,\ldots, N-1 \hbox{ and } g_1
\ldots g_N=g  \}.
\end{equation}

An admissible sequence $(g_{1}, \dots , g_{N}) \in {\mathcal C}^N_g$
is a solution of \emph{the discrete Euler-Lagrange equations} if
\[
0=\sum_{k=1}^{N-1}\left[\lvec{X}_k\big({g_k})(\mathbb{L}_d)-\rvec{X}_k\big({g_{k+1}})(\mathbb{L}_d)
\right], \; \; \mbox{ for } X_{1}, \dots , X_{N-1} \in
\Gamma(\tau_{AG}).
\]
For $N=2$ we obtain that $(g, h)\in G_2$ is a solution of the discrete Euler-Lagrange equations if
\[
\lvec{X}({g})(\mathbb{L}_d)-\rvec{X}({h})(\mathbb{L}_d)=0
\]
for every section $X$ of $AG$.

\begin{remark}
Marrero et al. \cite{MMM} showed that these discrete Euler-Lagrange equations are also equivalent to the sequence $g_1,\ldots,g_N\in G$ corresponding to a critical point of the action sum $$(g_1,\ldots,g_N)\longmapsto\sum_{k=1}^{N}\mathbb{L}_{d}(g_k),$$ over the space of admissible sequences.  

In the case when $G$ is the banal groupoid $Q\times Q\rightrightarrows Q$, this recovers the discrete Euler-Lagrange equations, $$D_{1}\mathbb{L}_{d}(q_k,q_{k+1})+D_{2}\mathbb{L}_{d}(q_{k-1},q_k)=0$$ for $k=1,\ldots,N-1,$ as in Marsden and West \cite{MaWest}.\hfill$\diamond$
\end{remark}

\subsubsection{Discrete Lagrangian evolution operator}

We say that a differentiable mapping $\Psi: G\rightarrow G$ is a
\emph{discrete flow} or a \emph{discrete Lagrangian evolution
operator for $\mathbb{L}_{d}$} if it verifies the following properties:
\begin{enumerate}
\item[-] $\hbox{graph}(\Psi)\subseteq G_2$, that is, $(g, \Psi(g))\in G_2$, $\forall g\in
G$.
\item[-] $(g, \Psi(g))$ is a solution of the discrete Euler-Lagrange
equations, for all $g\in G$, that is,
\begin{equation}\label{5.22'}
\lvec{X}(g)(\mathbb{L}_d)-\rvec{X}(\Psi(g))(\mathbb{L}_d)=0
\end{equation} for every section $X$ of $AG$ and every $g\in G.$
\end{enumerate}

\subsubsection{Discrete Legendre transformations} \label{section5.6}
Given a discrete Lagrangian $\mathbb{L}_{d}:G\to\R$ we define two
\emph{discrete Legendre transformations} $\F^{-}\mathbb{L}_{d}:
G\rightarrow A^*G$ and $\F^{+}\mathbb{L}_{d}: G\rightarrow A^*G$
 as follows (see \cite{MMM})
\begin{align}\label{DLt-}
(\F^{-}\mathbb{L}_{d})(h)(v_{\epsilon(\alpha(h))})=&-v_{\epsilon(\alpha(h))}(\mathbb{L}_d\circ
r_h\circ i), \mbox{ for } v_{\epsilon(\alpha(h))}\in
A_{\alpha(h)}G,\\
(\F^{+}\mathbb{L}_d)(g)(v_{\epsilon(\beta(g))})=&
v_{\epsilon(\beta(g))}(\mathbb{L}_d\circ \ell_g), \mbox{ for }
v_{\epsilon(\beta(g))}\in A_{\beta(g)}G.
\end{align} Note that  $(\F^{+}\mathbb{L}_d)(g)\in A^*_{\beta(g)}G$ and
$(\F^{-}\mathbb{L}_d)(h)\in A^*_{\alpha(h)}G$.

\subsubsection{Regular discrete Lagrangians and Hamiltonian evolution operator}

A discrete Lagrangian $\mathbb{L}_d:{G}\to \R$ is said to be \emph{regular} if and only if the Legendre transformation
$\F^+\mathbb{L}_d$  is a local diffeomorphism (equivalently, if and only if the Legendre transformation $\F^-\mathbb{L}_d$ is a local
diffeomorphism). In this case, if $(g_0,h_0)\in G\times G$ is a solution of the discrete Euler-Lagrange
equations for $\mathbb{L}_d$ then, one may prove (see \cite{MMM}) that there exist two open subsets $U_{0}$ and $V_{0}$ of
$G$, with $g_{0} \in U_{0}$ and $h_{0} \in V_{0}$, and there exists a (local) discrete Lagrangian evolution
operator $\Psi: U_{0} \to V_{0}$ such that:
\begin{enumerate}
\item $\Psi(g_{0}) = h_{0}$, \item $\Psi$ is a diffeomorphism and \item $\Psi$ is
unique, that is, if $U'_{0}$ is an open subset of $G$, with $g_{0} \in U_{0}',$ and $\Psi': U'_{0} \to
G$ is a (local) discrete Lagrangian evolution operator then
\[
\Psi\mid_{U_{0}\cap U_{0}'} = \Psi'\mid_{U_{0}\cap U_{0}'}.
\]
\end{enumerate}
Moreover, if $\F^{+}\mathbb{L}_d$ and $\F^{-}\mathbb{L}_d$ are global diffeomorphisms (that is, $\mathbb{L}_d$ is \emph{hyperregular}) then
$\Psi=(\F^{-}\mathbb{L}_d)^{-1}\circ \F^{+}\mathbb{L}_d$.

If $\mathbb{L}_d: {G}\to \R$ is a hyperregular Lagrangian function, then pushing forward to $A^*{G}$ with the discrete
Legendre transformations, we obtain the \emph{discrete  Hamiltonian evolution operator}, $\widetilde{\Psi}:
A^*{G}\to A^*{G}$ given by
\begin{equation}\label{dheo}
\widetilde{\Psi}=\F^{\pm}\mathbb{L}_d\circ \Psi\circ
(\F^{\pm}\mathbb{L}_d)^{-1}\;=\F^{+}\mathbb{L}_d\circ (\F^{-}\mathbb{L}_d)^{-1}.
\end{equation}

\subsubsection{Example: Discrete Euler-Poincar\'e equations}
Let $G$ be a Lie groupoid over $Q=\{\mathfrak{e}\}$, the identity of $G$. Given $\xi\in {\frak g}$ we have
the left and right invariant vector fields
\[
\lvec{\xi}(g)=(T_{\frak e}\ell_g)(\xi)\hbox{ and }\rvec{\xi}(g)=(T_{\frak e}
r_g)(\xi), \; \; \mbox{ for } g \in G.
\]

Given a discrete Lagrangian $\mathbb{L}_d: G\rightarrow \R$ its discrete
Euler-Poincar\'e equations are
\[
(T_{\frak e}\ell_{g_k})(\xi)(\mathbb{L}_d)-(T_{\frak e}r_{g_{k+1}})(\xi)(\mathbb{L}_d)=0,\;
\hbox{ for all } \xi\in {\frak g} \mbox{ and } g_{k}, g_{k+1} \in
G,
\]
that is, $$(\ell_{g_k}^*\hbox{d}\mathbb{L}_d)({\frak e}) = (r_{g_{k+1}}^*\hbox{d}\mathbb{L}_d)({\frak e})$$ (see
\cite{BoSu,Mars3,MaPeSh}).


 A way to  discretize a continuous problem is by using a \emph{retraction map} $\tau: {\mathfrak g}\to G$, which is an analytic local diffeomorphism and maps a neighborhood $V\subseteq\mathfrak{g}$ of $0\in {\mathfrak g}$ to a neighborhood $U\subseteq G$ of the identity $\mathfrak{e}\in G$.
We have that  $\tau(\xi)\tau(-\xi)=\mathfrak{e}$ for all $\xi \in \mathfrak{g}$ (see \cite{Rabee}). 

The retraction map provides a local chart on the Lie group and it  is used to express a small change in the group configuration
through a unique Lie algebra element, namely $\xi_k=\tau^{-1}(g_{k}^{-1}g_{k+1})/h$, where $h>0$ is a small enough time step, $\xi_k\in\mathfrak{g}$, $h\xi_k\in V\subseteq\mathfrak{g}$ and $g_{k},g_{k+1}\in U\subseteq G$, i.e., if $\xi_{k}$ were regarded as an average velocity between $g_{k}$ and $g_{k+1}$, then $\tau$ is an approximation of the corresponding vector field on $G$ (see \cite{KM}). 

 To derive the discrete Euler-Poincar\'e equations, one uses the {\it left-trivialized} tangent retraction map $d\tau_{\xi}:\mathfrak{g}\to\mathfrak{g}$ and its inverse $d\tau_{\xi}^{-1}:\mathfrak{g}\to\mathfrak{g}$ defined by
\begin{align*}
T_{\xi}\,\tau(\eta)=&T_{\mathfrak{e}} \ell_{\tau(\xi)}(\hbox{d}\tau_{\xi}(\eta))\equiv \hbox{d}\tau_{\xi}(\eta)\tau(\xi),\\
T_{\tau(\xi)}\tau^{-1}((T_{\mathfrak{e}}\ell_{\tau(\xi)})\eta)=&\hbox{d}\tau_{\xi}^{-1}(\eta),
\end{align*} for $\eta\in\mathfrak{g}$ (see \cite{MunteKaas}, \cite{Rabee}) .

The Lie algebra $\mathfrak{g}$ is on itself a vector space, then it is natural to consider local coordinates on $\mathfrak{g}$. We will write
$\tau(h\xi)=g$ for a enough small time step $h>0$ such that $h\xi\subseteq V$ where $V$ is a local neighborhood of $0\in {\mathfrak g}$. Fixing a basis $\{e_{\gamma}\}$ of $\mathfrak{g}$ we induce coordinates $(y^{\gamma})$ on $\mathfrak{g}$. In these coordinates, a basis of left-invariant and right-invariant vector fields is 
\begin{align*}
\lvec{e_{\gamma}}(\eta)&=T_{\tau(h\eta)}\tau^{-1}(\tau(h\eta)e_{\gamma})=
\mbox{d}\tau^{-1}_{h\eta} ({\rm Ad}_{\tau(h\eta)}e_{\gamma}),\\
\rvec{e_{\gamma}}(\eta)&=T_{\tau(h\eta)}\tau^{-1}(e_{\gamma}\tau(h\eta))=
\mbox{d}\tau^{-1}_{h\eta}(e_{\gamma}),
\end{align*}
where $\eta\in {\mathfrak g}$.

Given a Lagrangian $l: {\mathfrak g}\rightarrow \R$, the   discrete Euler-Poincar\'e equations are: 
\begin{equation}\label{eq:EP}
\lvec{e_{\gamma}}(\eta_k)(l)-\rvec{e_{\gamma}}(\eta_{k+1})(l)=0,
\end{equation} that is, 

$$(\hbox{d}\tau^{-1}_{h\eta_{k}})^{*}\left(\frac{\partial l}{\partial\xi}(\eta_k)\right)-(\hbox{d}\tau^{-1}_{-h\eta_{k+1}})^{*}\left(\frac{\partial l}{\partial\xi}(\eta_{k+1})\right)=0,$$ (see \cite{MMM2}).


\section{Second-order variational problems on Lie groupoids and optimal control applications}\label{variational}

In this section, we discuss discrete second-order Lagrangian
mechanics using techniques of variational calculus on Lie groupoids (see
\cite{IMMM} and \cite{MMM} for first order variational calculus on
Lie groupoids) and we illustrate our results with some examples and applications in the theory of optimal control of mechanical systems.

\subsection{Second-order variational problems on Lie groupoids}
Let $G$ be a Lie groupoid with structural
maps $\alpha,\beta:G\to Q$, $\epsilon:Q\to G$, $i:G\to G$ and $m:G_2\to G$.
Denote by $\tau_{AG}: AG\to Q$ the Lie algebroid associated with the Lie groupoid $G$.
\begin{definition}
A \textit{discrete second-order Lagrangian} $L_d:G_2\to\R$ is a
differentiable function defined on the set of composable
elements describing the dynamics of the mechanical
system.
\end{definition}

We denote by $G^{4}$ the product $G\times G\times G\times G$. As in the first order case, fixed $g\in G$, we define the set of admissible sequences in $G^4$ with
values in $G$ by consider $N=4$ in \eqref{sequences}, that is, 
$$C^{4}_{g}=\{(g_1,g_2,g_3,g_4)\in G^{4}\mid
(g_{k},g_{k+1})\in G_2\hbox{ for } k=1,2,3 $$ $$\hbox{ with }
g_1\hbox{ and }g_4 \hbox{ fixed and } g_{1}g_2g_3g_{4}=g\}.$$

Given a tangent vector at the point $\bar{g}=(g_1,g_2,g_3,g_4)$ to the
manifold $C^{4}_{g}$, we may write it as the tangent
vector at $t=0$ of a curve $c(t)$ in $C^{4}_{g}$, which passes through
$\bar{g}$ at $t=0$. This type of curves has the form
\begin{equation}\label{admisible}c(t)=(g_1,g_2h_2(t),h_2^{-1}(t)g_3,g_4),\end{equation} where $h_2(t)\in\alpha^{-1}(\beta(g_2))$, for all $t$, and
$h_2(0)=\epsilon(\beta(g_2)).$ The curve $c$ is called a
\textit{variation} of $\bar{g}$. Therefore we may identify the tangent space to $C^{4}_{g}$ at
$\bar{g}$ with
$$T_{\bar{g}}C^{4}_{g}\equiv\{v_2\mid
v_2\in A_{x_2}G \hbox{ where } x_2=\beta(g_2) \}.$$ The curve $v_2$
is called \textit{infinitesimal variation} of $\bar{g}$ and is the tangent vector to the $\alpha$-vertical curve $h_2$ at $t=0$.

 We define the \textit{discrete action sum}
associated with the discrete second-order Lagrangian $L_d:G_2\to\R$ by
\begin{eqnarray}\label{accion discreta}
S_{L_d}:C_{g}^{4}&\longrightarrow&\R\nonumber\\
\bar{g}&\longmapsto& \sum_{k=1}^{3}L_d(g_k,g_{k+1}).
\end{eqnarray}

 To derive the discrete equations of motion we apply Hamilton's principle of critical action. In order to do that, we need to consider the variations of the discrete action sum.

\begin{definition}[Discrete Hamilton's principle on Lie groupoids]\label{hamprinc}

Given $g\in G$, an admissible sequence
$\bar{g}\in C_g^{4}$ is a solution of the
Lagrangian system determined by $L_d:G_2\to\R$ if and only if
$\bar{g}$ is a critical point of $S_{L_d}$.
\end{definition}

\begin{proposition}
Given $g\in G$, the admissible sequence $\bar{g}=(g_1,g_2,g_3,g_4)\in C_g^{4}$ is a solution of the
Lagrangian system determined by $L_d:G_2\to\R$ if and only if $\bar{g}$ satisfies the \textit{discrete second-order Euler-Lagrange equations} for $L_d:G_2\to\R$ given by  
\begin{equation}\label{delelg}\ell^{*}_{g_{2}}\left(D_1L_d(g_{2},g_{3})+D_2
L_d(g_{1},g_{2})\right)+(r_{g_{3}}\circ
i)^{*}\left(D_1L_d(g_{3},g_{4})+D_2 L_d(g_2,g_{3})\right)=0.
\end{equation}
\end{proposition}

\textit{Proof.}  By definition \eqref{hamprinc}, $\bar{g}$ is a solution of the Lagrangian system determined by $L_d:G_2\to\R$ if it is a critical point of $S_{L_d}$. In order to characterize the critical points, we calculate,
$$\frac{d}{dt}\Big{|}_{t=0}S_{L_d}(c(t))=\frac{d}{dt}\Big{|}_{t=0}\{L_d(g_1,g_2h_2(t))+L_d(g_2h_2(t),h_2^{-1}(t)g_3)
+L_d(h_2^{-1}(t)g_3,g_4)\},$$ where $c(t)$ is the variation of $\bar{g}$ defined in \eqref{admisible}.

Then, the condition $\displaystyle{\frac{d}{dt}\Big{|}_{t=0}S_{L_d}(c(t))=0}$ is equivalent to
\begin{align}\label{condiferencial}
0=&d (L_d\circ \ell_{g_2})(\epsilon(\beta(g_2)))(v_2)+ d (L_d\circ r_{g_3}\circ i)(\epsilon(\beta(g_2)))(v_2)\nonumber\\
&+d (L_d\circ r_{g_3}\circ
i)(\epsilon(\beta(g_3)))(v_2)+ d (L_d\circ
\ell_{g_2})(\epsilon(\beta(g_1)))(v_2),
\end{align} where $v_2\in A_{\beta(g_2)}G$ is the infinitesimal variation of $\bar{g}$, $\epsilon(\beta(g_2))=h_2(0)$ and $\ell_{g}$ and $r_g$ were defined on Definition \ref{leftrighttranslationgroupoids}.


Therefore, $\bar{g}$ is a solution of the Lagrangian system determined by
the discrete second-order Lagrangian $L_d:G_2\to\R$ if and only if it satisfies equations \eqref{condiferencial}, that is, $\bar{g}$ is a solution of the Lagrangian system
determined by $L_d:G_2\to\R$ if and only if $\bar{g}$ satisfies
\begin{equation}\label{delelg}\ell^{*}_{g_{2}}\left(D_1L_d(g_{2},g_{3})+D_2
L_d(g_{1},g_{2})\right)+(r_{g_{3}}\circ
i)^{*}\left(D_1L_d(g_{3},g_{4})+D_2 L_d(g_2,g_{3})\right)=0.
\end{equation}
\hfill$\square$

The  equations given above are called \textit{discrete second-order
Euler-Lagrange equations}.

\subsubsection{Example: Discrete second-order Euler-Lagrange equations on the pair groupoid.}
Let $Q\times Q\rightrightarrows Q$ be the Banal groupoid. An
admissible path is the $4$-tuple
$((q_1,q_2),(q_2,q_3),(q_3,q_4),(q_4,q_5))\in
C^{4}_{(q_1,q_5)}$.
$(Q\times Q)_{2}$  is isomorphic to $3Q$ where inclusion of $3Q$ into $(Q\times Q)_{2}$ is given by the map
$(q,\tilde{q},\bar{q})\hookrightarrow((q,\tilde{q}),(\tilde{q},\bar{q}))$.
Applying Hamilton's principle for the discrete second-order Lagrangian  $L_d:(Q\times Q)_{2}\simeq 3Q\to\R$
given by $L_{d}((q_k,q_{k+1}),(q_{k+1},q_{k+2}))=L_d(q_k,q_{k+1},q_{k+2})$, one gets
$$\frac{d}{dt}\Big{|}_{t=0}S_{L_d}=\frac{d}{dt}\Big{|}_{t=0}\sum_{k=1}^{3}L_d(q_k,q_{k+1},q_{k+2}).$$
Hence, the path $((q_1,q_2),(q_2,q_3),(q_3,q_4),(q_4,q_5))\in
C^{4}_{(q_1,q_5)}$  is a critical point of $S_{L_d}$ if and only if
it satisfies 
\begin{equation}\label{dsoel}D_3L_d(q_{1},q_{2},q_3)+D_2L_d(q_{2},q_{3},q_{4})+D_1L_d(q_3,q_{4},q_{5})=0.
\end{equation} for $(q_1,q_2)$ and $(q_4,q_5)$ fixed points in the Banal groupoid.

These equations are the \textit{discrete second-order Euler-Lagrange
equations} for $L_d:Q\times Q\times Q\to\R$ (see for example \cite{belema}).

\subsubsection{Example: Discrete second-order Euler-Poincar\'e equations} Let $G$ be a Lie group, that is $G$ is a Lie groupoid over the identity element  $\{\mathfrak{e}\}$ of $G$. Given $g\in G$
and a discrete second-order Lagrangian $L_d:G_2\to\R$, 
the solution for the Lagrangian system determined by the discrete Lagrangian $L_d$ are
\begin{align*}0=&\ell^{*}_ {g_{k}}D_{1}L_{d}(g_k,g_{k+1})+\ell^{*}_{g_{k}}D_{2}L_{d}(g_{k-1},g_{k})-r^{*}_{g_{k+1}}D_{2}L_{d}(g_{k},g_{k+1})\\&-r^{*}_{g_{k+1}}D_{1}L_{d}(g_{k+1},g_{k+2})\end{align*} for an admisible sequence $(g_{k-1},g_k,g_{k+1},g_{k+2})\in \mathcal{C}^4_{g}$ with $g=g_{k-1}g_kg_{k+1}g_{k+2}$ and $g_1,g_4$ fixed points in $G$. 

The equations given above are the \textit{discrete second-order Euler-Poincar\'e equations} (see for example \cite{CojiMdD} and \cite{BHM}).


\subsubsection{Example: Discrete second-order Euler-Lagrange equations on an action Lie groupoid}

Let $H$ be a Lie group, and $Q$ a differentiable manifold. Let $\varphi:Q\times H\to Q$ be a right action, $\varphi(q,h)=qh$. We consider the action Lie groupoid, $G=Q\times H$ over $Q$. The set of composable elements is determined by $$G_2=\{((q_1,h_1),(q_2,h_2))\mid q_{2}=q_1 h_1\}.$$ 

If $(q,h)\in G$, the left-translation $\ell_{(q,h)}:\alpha^{-1}(qh)\to\alpha^{-1}(q)$ and the right-translation $r_{(q,h)}:\beta^{-1}(q)\to\beta^{-1}(qh)$ (where $\alpha$ and $\beta$ are the source and target map of $G$) are given by $\ell_{(q,h)}(qh,h')=(q,hh')$ and $r_{(q,h)}(q(h')^{-1},h')=(q(h')^{-1},h'h)$ for  $q\in Q, h,h'\in H$.

Consider an admissible path $$((q_1,h_1),(q_1h_1,h_2),(q_1h_1h_2,h_3),(q_1h_1h_2h_3,h_4))\in\mathcal{C}_{x}^{4}$$ where $x=(q_1,h)\in Q\times H$ with $h=h_1h_2h_3h_4$  and $(q_1,h_1), (q_1h_1h_2h_3,h_4)$ are fixed in $(Q\times H)$. A discrete second-order Lagrangian is defined as $L_{d}:(Q\times H)_2\simeq Q\times H\times H\to\R$ by $$L_{d}(q_1,h_1,h_2):=L_{d}((q_1,h_1),(q_1h_1,h_2)).$$ The discrete Euler-Lagrange equations for the system determined by the discrete second-order discrete Lagrangian $L_{d}:Q\times H\times H\to\R$ are determined by 

\begin{align*}
0=&\ell_{(q_1h_1,h_{2})}^{*}\left(D_1L_{d}(q_1h_1,h_{2},h_{3})+D_{2}L_d(q_1h_1,h_{2},h_{3})+D_3L_d(q_1,h_1,h_2)\right)\\
&+(r_{(q_1h_1h_2,h_{3})}\circ i)^{*}\left(D_1L_d(q_1h_1h_2,h_{3},h_{4})+D_2L_{d}(q_1h_1h_2,h_{3},h_{4})\right.\\
&\left.+D_{3}L_d(q_1h_1,h_{2},h_{3})\right).\\
q_2=&\varphi(q_{1},h_1).
\end{align*}

 \subsection{Second-order constrained variational problems on Lie groupoids}\label{constrainedcase}
Next, we extend the previous variational principle to second-order variational problems for systems subject to second-order constraints. The constructions presented here are  interesting for applications in optimal control problem of underactuated mechanical controlled systems.

Let $L_d:G_2\to\R$ be a discrete second-order Lagrangian describing
the dynamics of a discrete mechanical system. Suppose that the dynamics is
restricted. This restriction is given by the vanishing of $s$ smooth
constraint functions $\Phi_d^{A}:G_2\to\R,$ $A=1,\ldots,s$ determining a submanifold $\mathcal{M}$ of $G_2$.

The dynamics of the second-order constrained variational problem associated with $L_d$ and $\Phi_d^{A}$ is described by the discrete constrained second-order Euler-Lagrange equations determined by considering the augmented Lagrangian $\widehat{L}_{d}:G_{2}\times\R^{s}\to\R$ given by 
$\widehat{L}_d=L_d+\lambda_{A}\Phi_d^{A}$
 where $\lambda_{A}=(\lambda_{1},\ldots,\lambda_{s})\in\R^{s}$ are Lagrange multipliers to be determined (see
subsection \ref{discreteconstrained} for an intrinsic approach).

Given $g\in G$, the set of admissible sequences is given by  $$C^{4}_{g,\mathcal{M}}=\{(g_1,g_2,g_3,g_4,\lambda^1,\lambda^2,\lambda^{3})\in G^{4}\times\R^{3s}\mid
(g_{k},g_{k+1})\in G_2, \Phi_{d}^{A}(g_k,g_{k+1})=0 $$ $$ \hbox{ for } k=1,2,3  \hbox{ with }
g_1\hbox{ and }g_4 \hbox{ fixed and } g_{1}g_2g_3g_{4}=g \}.$$

Consider the \textit{extended action sum} associated with the extended Lagrangian $\widehat{L}_{d}:G_2\times\R^{s}\to\R$ 
\begin{eqnarray}\label{accion discreta extended}
S_{\widehat{L}_d}:C_{g,\mathcal{M}}^{4}&\longrightarrow&\R\nonumber\\
(g_1,g_2,g_3,g_4,\lambda^1,\lambda^2,\lambda^3)&\longmapsto& \sum_{k=1}^{3}\left[L_d(g_k,g_{k+1})+(\lambda^{k}_A)^{T}\Phi^{A}_{d}(g_k,g_{k+1})\right]\; ,
\end{eqnarray} where $\lambda_{
A}^{k}=(\lambda_{1}^{k},\ldots,\lambda_{s}^{k})\in\R^{s}$. An easy adaptation of the variational principle \eqref{hamprinc} for the discrete extended Lagrangian $\widehat{L}_d$
can be done to obtain the discrete constrained second-order Euler-Lagrange
equations by extremizing the extended action sum $S_{\widehat{L}_d}$. The equations describing the dynamics of second-order constrained variational problems are
\begin{align}\label{constrainedso}
 0=&\Phi_d^{A}(g_{1},g_{2}),\quad 0=\Phi_d^{A}(g_{2},g_{3}),\quad 0=\Phi_d^{A}(g_{3},g_{4}) \hbox{ with }A=1,\ldots,s,\nonumber\\
0=&\ell_{g_{2}}^{*}\left(D_1L_{d}(g_{2},g_{3})+\lambda^{2}_{A}D_{1}\Phi^{A}_d(g_{2},g_{3})
+D_{2}L_{d}(g_{1},g_{2})+\lambda^{1}_{A}D_{2}\Phi^{A}_d(g_{1},g_{2})\right)\\
&+(r_{g_{3}}\circ i)^{*}\left(D_1L_{d}(g_{3},g_{4})+\lambda^{3}_{A}D_{1}\Phi^{A}_d(g_{3},g_{4})+D_{2}L_{d}(g_{2},g_{3})+\lambda^{2}_{A}D_{2}\Phi^{A}_d(g_{2},g_{3})\right).\nonumber
\end{align}

\subsection{Application to optimal control of mechanical systems}

In this section we study how to apply the second-order Euler-Lagrange equations on Lie groupoids to optimal control problems of mechanical systems defined on Lie algebroids. After introducing  optimal control control problems, we study their discretization.

\subsubsection{Optimal control problems of total-actuated mechanical systems on Lie algebroids}

Let  $(A,\llbracket\cdot,\cdot\rrbracket,\rho)$ be a Lie algebroid over
$Q$ with bundle projection $\tau_{A}:A\to Q$. The dynamics is specified fixing a Lagrangian $L:A\to\R$ (see \hyperref[ApC]{Appendix C}) . External forces are modeled,
in this case, by curves $u_F:\R\to A^{*}$ where $A^{*}$ is the dual bundle $\tau_{A^{*}}:A^{*}\to\ Q$.  

Given local coordinates $(q^i)$ on $Q$, and fixing a basis of sections $\{e_{\alpha}\}$ of $\tau_{A}:A\to Q$ we can
induce local coordinates $(q^i,y^{\alpha})$ on $A$; that is, every element $b\in A_q=\tau_A^{-1}(q)$ is expressed univocally as $b=y^{\alpha}e_{\alpha}(q)$.  The notion of admissible
curves replaces that of natural prolongation in the context of Lie algebroids.

\begin{definition}
Let  $(A,\llbracket\cdot,\cdot\rrbracket,\rho)$ be a Lie algebroid over
$Q$ with projection $\tau_{A}:A\to Q$.  A curve $\xi:I\subset\mathbb{R}\to A$ is an
\textit{admissible curve} on $A$ if 
$$\rho(\xi(t))=\frac{d}{dt}(\tau_{A}(\xi(t))).$$

In a local description, a curve $\xi$ on $A$ given by $\xi(t)=(q^{i}(t),y^{\alpha}(t))$, is admissible if $$\dot{q}^{i}=\rho_{\alpha}^{i}(q)y^{\alpha}$$
where if $b=y^{\alpha}e_{\alpha}(q)$ with $q=\tau_{AG}(b)$ then $\displaystyle{\rho(b)=\rho_{\alpha}^{i}(q)y^{\alpha}\frac{\partial}{\partial q^{i}}\Big{|}_{q}}$.
\end{definition}

It is possible to adapt the derivation of the Lagrange-d'Alembert
principle to study fully-actuated mechanical controlled systems
on Lie algebroids (see \cite{CoMa} and \cite{Ma2}). Let $q_0$ and $q_T$ fixed in $Q$, consider an admissible curve $\xi:I\subset\mathbb{R}\to A$
which satisfies the principle
$$0=\delta\int_{0}^{T}L(\xi(t))dt+\int_{0}^{T}\langle u_{F}(t),\eta(t)\rangle dt,$$
where $\eta(t)\in {A}_{\tau_{A}(\xi(t))}$ and $u_{F}(t)\in A^{*}_{\tau_{A}(\xi(t))}$ defines
the control force (where we are assuming they are arbitrary). The infinitesimal variations in the variational principle are given by
$\delta\xi=\eta^{C}$, for all time-dependent sections
$\eta\in\Gamma(\tau_{A})$, with $\eta(0)=0$ and $\eta(T)=0$, where
$\eta^{C}$ is a time-dependent vector field on $A$, the
\textit{complete lift}, locally defined by
$$\eta^C=\rho_{\alpha}^i\eta^{\alpha}\frac{\partial}{\partial q^i}+
(\dot{\eta}+\mathcal{C}^{\alpha}_{\beta\gamma}\eta^{\beta}y^{\gamma})\frac{\partial}{\partial
y^{\alpha}}$$ (see \cite{CoMa},
\cite{Eduardo}, \cite{Eduardo1} and \cite{Eduardoalg}). Here the structure functions $\mathcal{C}_{\beta\gamma}^{\alpha}$ are determined by $\lcf e_{\beta},e_{\gamma}\rcf=\mathcal{C}_{\beta\gamma}^{\alpha}e_{\alpha}$.

From the Lagrange-d'Alembert
principle one easily derives the controlled
Euler-Lagrange equations by using standard variational calculus  \begin{align*}
\frac{d}{dt}\left(\frac{\partial L}{\partial
y^{\alpha}}\right)-\rho_{\alpha}^{i}\frac{\partial L}{\partial
q^{i}}+\mathcal{C}_{\alpha\beta}^{\gamma}(q)y^{\beta}\frac{\partial
L}{\partial y^{\gamma}}=&(u_{F})_{\alpha},\\
\frac{dq^i}{dt}=&\rho_{\alpha}^{i}y^{\alpha}.
\end{align*}

The control force $u_{F}$ is chosen such that it minimizes the
cost functional
$$\int_0^{T}C(q^i,y^{\alpha},(u_{F})_{\alpha})dt,$$ where
$C:A\oplus A^{*}\to\R$ is the cost function associated with the optimal control problem.

Therefore, the optimal control problem consists on finding an admissible curve $\xi(t)=(q^{i}(t),y^{\alpha}(t))$ solution of the controlled Euler-Lagrange equations,  the  boundary conditions and minimizing the cost functional for $C: A\oplus A^{*}\to\R$. This optimal control problem can be equivalently  solved as a second-order variational problem by defining the second-order Lagrangian $\widetilde{L}:A^{(2)}\to\R$ as
\begin{equation}\label{costLtilde}\widetilde{L}(q^i,y^{\alpha},\dot{y}^{\alpha})=C\left(q^i,y^{\alpha},\frac{d}{dt}\left(\frac{\partial
L}{\partial y^{\alpha}}\right)-\rho_{\alpha}^{i}\frac{\partial
L}{\partial
q^{i}}+\mathcal{C}_{\alpha\beta}^{\gamma}(q)y^{\beta}\frac{\partial
L}{\partial y^{\gamma}}\right).\end{equation} Here $A^{(2)}$ denotes the set of admissible elements of the Lie
algebroid $A$, a subset of $ A\times TA$, given by 
$$A^{(2)}:=\{(b,v_b)\in A\times TA\mid \rho(b)=T\tau_A (v_b)\}$$
where $T\tau_A:TA\to TQ$ is the tangent map of the bundle projection. $A^{(2)}$ is considered as the substitute of the second-order tangent bundle in classical mechanics \cite{Eduardoalg}. In local coordinates, the set $A^{(2)}$ is characterized by the tuple $(q^{i},y^{\alpha},z^{\alpha},v^{\alpha})\in A\times TA$ such that 
$y^{\alpha}=z^{\alpha}$. Therefore one can consider local coordinates $(q^{i},y^{\alpha},v^{\alpha})$ on $A^{(2)}$.

The dynamics associated with the second-order Lagrangian $\widetilde{L}:A^{(2)}\to\R$ (and therefore the optimality conditions for the optimal control problem) is given by the second-order Euler-Lagrange equations on Lie algebroids  (see for example \cite{LeoTesis} and \cite{Eduardoho})
\begin{equation}\label{ecuacionesvakonomoalgebroides2}
0=\frac{d^2}{dt^2}\left(\frac{\partial \widetilde{L}}{\partial
    v^{\alpha}}\right)+\mathcal{C}_{\alpha\beta}^{\gamma}(q)y^{\beta}\frac{d}{dt}\left(\frac{\partial
    \widetilde{L}}{\partial v^{\gamma}}\right)-\frac{d}{dt}\frac{\partial
    \widetilde{L}}{\partial
    y^{\alpha}}-\mathcal{C}_{\alpha\beta}^{\gamma}(q)y^{\beta}\frac{\partial
    \widetilde{L}}{\partial y^{\gamma}}+\rho_{\alpha}^{i}\frac{\partial
    \widetilde{L}}{\partial q^{i}}, 
\end{equation} together with the admissibility condition $\displaystyle{\frac{dq^i}{dt}=\rho_{\alpha}^{i}y^{\alpha}}.$

\begin{remark}
Alternatively, one can define the Lagrangian
$\widetilde{L}:A^{(2)}\to\R$ in terms of the Euler-Lagrange operator as
$$\widetilde{L}=C\circ(\tau^{\tiny{A^{(2)}}}_{A}\oplus\mathcal{EL}(L)):A^{(2)}
\to\R,$$ where $\mathcal{EL}(L):A^{(2)}\to A^{*}$ is the \textit{Euler-Lagrange operator} which locally reads as
$$\mathcal{EL}(L)=\left(\frac{d}{dt}\frac{\partial L}{\partial y^{\alpha}}-\rho_{\alpha}^{i}\frac{\partial
L}{\partial
q^{i}}+\mathcal{C}_{\alpha\beta}^{\gamma}(q)y^{\beta}\frac{\partial
L}{\partial y^{\gamma}}\right)e^{\alpha}.$$ Here $\{e^{\alpha}\}$ is the dual basis of $\{e_{\alpha}\},$ the basis of sections of $A$ and $\tau^{\tiny{A^{(2)}}}_{A}:A^{(2)}\to A$ is the canonical projection between $A^{(2)}$ and $A$ given by the map $A^{(2)}\ni (q^{i},y^{\alpha},v^{\alpha})\mapsto(q^{i},y^{\alpha})\in A.$\hfill$\diamond$
\end{remark}

\subsubsection{Optimal control problems of underactuated mechanical systems on Lie algebroids}
Now, suppose that our mechanical control system is underactuated, that is, the number of control inputs is less than the dimension of the configuration space.
The class of underactuated mechanical systems are abundant in real life for different reasons;
for instance, as a result of design choices motivated by the search of less cost engineering
devices or as a result of a failure regime in fully actuated mechanical systems. Underactuated
systems include spacecrafts, underwater vehicles, mobile robots, helicopters, wheeled vehicles
and underactuated manipulators.
We will see that the corresponding  optimality conditions are given by the solutions of second-order constrained Euler-Lagrange equations  (see \cite{CoMdD}). 

 
Given a Lagrangian function $L:A\to\R$ and control external forces, the controlled equations for an underactuated system defined on a Lie algebroid are \begin{align*}
\frac{d}{dt}\left(\frac{\partial L}{\partial
y^{a}}\right)-\rho_{a}^{i}\frac{\partial L}{\partial
q^{i}}+\mathcal{C}_{a\beta}^{\gamma}(q)y^{\beta}\frac{\partial
L}{\partial y^{\gamma}}=&(u_{F})_{a},\\
\frac{d}{dt}\left(\frac{\partial L}{\partial
y^{A}}\right)-\rho_{A}^{i}\frac{\partial L}{\partial
q^{i}}+\mathcal{C}_{A\beta}^{\gamma}(q)y^{\beta}\frac{\partial
L}{\partial y^{\gamma}}=&0,\\
\frac{dq^i}{dt}=&\rho_{\alpha}^{i}y^{\alpha},
\end{align*} with $\{\alpha\}=\{A,a\}$. The optimal control problem consists on finding an admissible trajectory 
$\xi(t)=(q^{i}(t),y^{A}(t),y^{a}(t))$ solution of the controlled Euler-Lagrange equations given boundary conditions and minimizing a cost functional  $C:A\oplus A^{*}\to\R$.

This optimal control problem can be solved as a constrained second-order variational problem on Lie algebroids where the second-order Lagrangian $\widetilde{L}:A^{(2)}\to\R$ is given by
\begin{equation}\label{costLtilde}\widetilde{L}(q^i,y^{\alpha},v^{a}, v^{A})=C\left(q^i,y^{\alpha},\frac{d}{dt}\left(\frac{\partial L}{\partial
y^{a}}\right)-\rho_{a}^{i}\frac{\partial L}{\partial
q^{i}}+\mathcal{C}_{a\beta}^{\gamma}(q)y^{\beta}\frac{\partial
L}{\partial y^{\gamma}}\right)\end{equation} and where the dynamics is restricted by the second order constraints $$\Phi^{A}(q^{i},y^{\alpha},v^{a},v^{A})=\frac{d}{dt}\left(\frac{\partial L}{\partial
y^{A}}\right)-\rho_{A}^{i}\frac{\partial L}{\partial
q^{i}}+\mathcal{C}_{A\beta}^{\gamma}(q)y^{\beta}\frac{\partial
L}{\partial y^{\gamma}}=0.$$

The optimality conditions for the optimal control problem are determined by the second-order constrained Euler-Lagrange equations given by considering the extended Lagrangian $\widehat{L}=\widetilde{L}+\lambda_{A}\Phi^{A}:A^{(2)}\times\R^{s}\to\R$ where $\lambda_{A}=(\lambda_1,\ldots,\lambda_s)\in\R^s$ are the Lagrange multipliers. These equations are given by  (see \cite{LeoTesis} for more details)

\begin{align*}
0=&
\frac{d^2}{dt^2}\left(\frac{\partial \widetilde{L}}{\partial
    v^{\alpha}}
  + {\lambda}_{A}\frac{\partial\Phi^{A}}{\partial v^{\alpha}}
    \right)+\mathcal{C}_{\alpha\beta}^{\gamma}(q)y^{\beta}\frac{d}{dt}\left(\frac{\partial
    \widetilde{L}}{\partial v^{\gamma}}+ {\lambda}_{A}\frac{\partial\Phi^{A}}{\partial v^{\gamma}}\right)-\frac{d}{dt}\frac{\partial
    \widetilde{L}}{\partial
    y^{\alpha}}\\
    &
    -{\lambda}_{A}\frac{\partial\Phi^{A}}{\partial y^{\alpha}}
    -\mathcal{C}_{\alpha\beta}^{\gamma}(q)y^{\beta}\left(\frac{\partial
    \widetilde{L}}{\partial y^{\gamma}}+ {\lambda}_{A}\frac{\partial\Phi^{A}}{\partial y^{\gamma}}\right)+\rho_{\alpha}^{i}\left(\frac{\partial
    \widetilde{L}}{\partial q^{i}}+ {\lambda}_{A}\frac{\partial\Phi^{A}}{\partial q^{i}}\right), 
\\
0=&\Phi^{A}(q^{i},y^{\alpha},v^{\alpha})
\end{align*} together with the admissibility condition $\dot{q}^{i}=\rho_{\alpha}^{i}y^{\alpha}$


\subsubsection{Optimal control problems on Lie groupoids}

Now we describe the discrete optimal control problem on a Lie groupoid $G$. Let $\mathbb{L}_d:G\to\R$ be a discrete Lagrangian, an
approximation of the action corresponding to a  continuous Lagrangian $L:A\to\R$ defined on a Lie algebroid $A$, that is, 
$$\mathbb{L}_d(g_k)\simeq\int_{kh}^{(k+1)h}L(\xi(t))dt$$ where $h>0$ is the time step with $T=Nh$ and $\xi$ is an admissible curve on $A$.

The discrete controlled Euler-Lagrange equations are
\begin{equation}\label{discretecontroleq}
\ell_{g_k}^{*}\hbox{d}\mathbb{L}_d(g_k)-(r_{g_{k+1}}\circ
i)^{*}\hbox{d}\mathbb{L}_d(g_{k+1})=u_k\in A^{*}_{\beta(g_{k})}G,
\end{equation} for $k=1,\ldots,N-1$ where $g_0$ and $g_{N}$ are fixed on $G$.

We define the subset $G_{\beta}\times_{\tau_{A^{*}G}}A^{*}G$ of $G\times A^{*}G$, 
$$G_{\beta}\times_{\tau_{A^{*}G}}A^{*}G:=\{(g,u)\in G\times
A^{*}G\mid\beta(g)=\tau_{A^{*}G}(u)\}.$$ Given a discrete cost function $C_{d}:G_{\beta}\times_{\tau_{A^{*}G}}A^{*}G\to\R$, the discrete optimal control problem is determined by extremizing the
discrete cost functional
\begin{equation}\label{jd}\mathcal{J}_d(g_0,g_1,\ldots,g_{N}):=\sum_{k=0}^{N-1}C_{d}(g_k,u_k),\end{equation} 
for $(g_0,g_1,\ldots,g_{N})\in G^{N+1}$, satisfying equations
\eqref{discretecontroleq} with $(g_k,g_{k+1})\in G_2$,
$k=0,\ldots, N-1$ and with $g_0,g_1,g_{N-1},g_{N}$ and $g=g_0g_1\ldots
g_{N}\in G$ fixed points on $G$. 

Defining the discrete second-order Lagrangian
$L_d:G_2\to\R$ as \begin{equation}\label{lagordendos}
L_{d}(g_k,g_{k+1}):=C_{d}\left(g_k,\ell_{g_k}^{*}\hbox{d}\mathbb{L}_d(g_k)-(r_{g_{k+1}}\circ
i)^{*}\hbox{d}\mathbb{L}_d(g_{k+1})\right),
\end{equation} the discrete optimal control problem consists on finding a path
$(g_0,g_1,\ldots,g_N)\in G^{N+1}$ minimizing the discrete
action sum $\mathcal{J}_{d}$ for the discrete second-order
Lagrangian $L_d:G_2\to\R$ where $g_0,g_1,g_{N-1},g_N$ and
$g=g_0g_1\ldots g_N\in G$ are fixed points in $G.$

Discrete Hamilton's principle \eqref{hamprinc} states that the paths
minimizing $\mathcal{J}_d$ subject fixed points
$g_0,g_1,g_{N-1},g_N\in G$ satify the discrete second-order
Euler-Lagrange equations for $L_d:G_2\to\R$ given by
\begin{align*}\label{doptimalcontrol}
0=&\ell^{*}_{g_{k}}\left(D_1L_d(g_{k},g_{k+1})+D_2
L_d(g_{k-1},g_{k})\right)\\&+(r_{g_{k+1}}\circ
i)^{*}\left(D_1L_d(g_{k+1},g_{k+2})+D_2L_d(g_k,g_{k+1})\right).
\end{align*} Therefore, as in the continuous problem, the optimality conditions for the discrete optimal control problem are determined by the discrete second-order
Euler-Lagrange equations for $L_d:G_2\to\R$.

Alternatively, one can start with a continuous optimal control problem associated to a Lagrangian $L:A\to\R$ defined on a Lie algebroid $A$. The optimality conditions for this optimal control problem are determined by a system of fourth order  differential equations obtained from the second-order Euler-Lagrange equations associated with the Lagrangian $\widetilde{L}:A^{(2)}\to\R$ determined by the cost function as in \eqref{costLtilde}. Now, we take directly  a  discretization of the second-order Lagrangian $\widetilde{L}$ to derive $L_d:G_2\to\R$.

Finally, we would like to point out that the underactuated case follows as in the continuous case by consider a discrete second-order constrained problem as in Subsection \ref{constrainedcase} and the optimality conditions are given by the solutions of the discrete second-order constrained Euler-Lagrange equations \eqref{constrainedso}. We will illustrate this in Example \ref{rotors}.
\subsubsection{An illustrative example: Optimal control of a rigid body on SO(3)}

In this example, we show how the optimal control problem of a rigid body defined on the Lie
group $SO(3)$  can be studied using the previous constructions given before.  This example is motivated by the attitude optimal control of spacecrafts (see \cite{Leok}, \cite{S9} and references therein).

 The Lie groupoid structure of $SO(3)$ over the identity matrix $\hbox{Id}$ is given by
$$\alpha(R)=Id,\quad\beta(R)=Id,\quad \epsilon(Id)=Id,\quad i(R)=R^{-1}\hbox{ and
}m(RG)=RG$$ for $R,G\in SO(3).$
The Lie algebroid associated with the Lie groupoid $SO(3)$ is the
Lie algebra $\mathfrak{so}(3)$ over a single point, where the anchor map is zero and the bracket is the usual commutator of matrices and the set of admissible elements is identified with $\mathfrak{so}(3)\times\mathfrak{so}(3)$.  Observe that, in this case, all the elements are composable, that is, $SO(3)_{2}=SO(3)\times SO(3)$.

The equations of motion of the controlled rigid body are
\begin{equation}\label{CuerpoRigido}
\dot\Om_{1}=P_1\Om_{2}\Om_{3}+u_{1},\quad
\dot\Om_{2}=P_{2}\Om_{1}\Om_{3}+u_{2},\quad \dot\Om_{3}=P_{3}\Om_{1}\Om_{2}+u_{3},
\end{equation}
where $\Omega=(\Om_{1}, \Om_{2},\Om_{3})\in\R^{3}$ and
$\dot\Om=(\dot\Om_{1}, \dot\Om_{2},\dot\Om_{3})\in\R^3$, $u_{i}$ are
the control inputs (torques for the rigid body) for $i=1,2,3$ and
$$P_1=\frac{I_1}{I_2-I_3},\quad P_2=\frac{I_2}{I_3-I_1},\quad P_3=\frac{I_3}{I_1-I_2},$$
are constants determined by the moments of inertia of the rigid  body $I_{1},I_2,I_3$. Here, we are using the typical identification of the Lie algebra
 $\mathfrak{so}(3)$ with $\R^3$ by the hat map
$\hat\cdot:\R^{3}\to\mathfrak{so}(3)$ (see \cite{Blo} and \cite{holmbook} for
example), where with some abuse of notation, we directly identify
$\R^{3}$ with $\mathfrak{so}(3)$ by omitting the hat notation.

Our fixed boundary conditions for the optimal control problem are $(R(0),\Om(0))$ and
$(R(T),\Om(T))$, where $R(t)\in SO(3)$ is the attitude of the rigid
body subject to the reconstruction equation $\dot R=R\Om$ and variations for the attitude are given by $\delta R=R\eta$,
with $\eta$ an arbitrary curve on $\mathfrak{so}(3)$. Consider the cost functional
\[
\mathcal{J}=\frac{1}{2}\int^{T}_{0}\left(u_1^{2}+u_2^{2}+u_3^{3}\right)\,dt.
\]
From eqs. (\ref{CuerpoRigido}) we can work out $u_1$, $u_2$ and $u_3$ in terms of $\Om$
and $\dot\Om$. Consequently, we can define the function
$\widetilde{L}:\mathfrak{so}(3)\times\mathfrak{so}(3)\to\R$ by $\widetilde{L}(\Om,\dot\Om)=\frac{1}{2}u(\Om,\dot\Om)\cdot u(\Om,\dot\Om)$
where $u=(u_{1},u_{2},u_{3})$. Therefore, the Lagrangian function
$\widetilde{L}:\mathfrak{so}(3)\times\mathfrak{so}(3)\to\R$ is 
\begin{equation}\label{LagCont}
\widetilde{L}(\Om,
\dot\Om)=\frac{1}{2}\left(\left(\dot\Om_{1}-P_{1}\Om_{2}\Om_{3}\right)^{2}+\left(\dot\Om_{2}-P_{2}\Om_{1}\Om_{3}\right)^{2}
+\left(\dot\Om_{3}-P_{3}\Om_{1}\Om_{2}\right)^{2}\right).
\end{equation}
From \eqref{LagCont} the cost functional becomes into
$\displaystyle{
\mathcal{J}=\int^{T}_{0}\widetilde{L}(\Om,\dot\Om)\,dt,
}$  (see \cite{CoMdD} for the solution of this second-order variational problem in the continuous setting).

Angular velocities  and angular accelerations can be approximated by discrete trajectories $\xi_{k+1/2}\simeq\Omega(kh)$ and 
$\xi_{k,k+1}\simeq\dot{\Omega}(kh)$ respectively for $\xi_k, \xi_{k+1}\in\alg$ where $h>0$ is a fixed real number and $T=kh$ with $k=0,\ldots,N$, where we are using the notation $\xi_{k+1/2}=\frac{1}{2}(\xi_k+\xi_{k+1})$ and $\xi_{k,k+1}=\frac{1}{h}(\xi_{k+1}-\xi_k)$.  Define the second-order Lagrangian $\mathcal{L}:\alg\times\alg\to\R$ by
 $\mathcal{L}(\xi_k,\xi_{k+1})=\widetilde{L}\left(\xi_{k+1/2},\xi_{k,k+1}\right)$.

The optimal control
problem, is given by minimizing the cost function associated with the discrete second-order Lagrangian 
$L_d:SO(3)\times SO(3)\to\R$ over discrete paths on $SO(3)$ where
\begin{equation}\label{sd}
L_d(w_k,w_{k+1})=h\mathcal{L}(\ca^{-1}(w_k),\ca^{-1}(w_{k+1})).\end{equation}
Here $(w_k,w_{k+1})\in SO(3)\times SO(3)$, $\ca(h\xi_k)=w_k$ and $\ca:\alg\to SO(3)$ denote
the Cayley map for the Lie group $SO(3)$ (see \hyperref[ApD]{Appendix D}).
 
Therefore, the discrete Lagrangian is now given by 
\begin{equation*}\label{sdfinal}
L_d(w_k,w_{k+1})=h\widetilde{L}\left(\frac{\ca^{-1}(w_k)+\ca^{-1}(w_{k+1})}{2h},
\frac{\ca^{-1}(w_{k+1})-\ca^{-1}(w_k)}{h^2}\right).
\end{equation*}
%

The variational integrator for the optimal control problem is given by applying discrete  Hamilton's principle
\eqref{hamprinc} in the discrete action sum determined by the discrete 
cost 
\begin{equation}\label{coste}\mathcal{C}_d=\sum_{k=0}^{N-1}L_d(w_k,w_{k+1}).\end{equation}

Instead of working  with the discrete sum \eqref{coste} one can take
\begin{equation}\label{coste2}\mathcal{C}_d=\sum_{k=0}^{N-1}\widetilde{L}\left(\xi_{k+1/2},\xi_{k,k+1}\right)\end{equation}
in order to work in a vector space where variations of $\xi_k=\ca^{-1}(w_k)\in\alg$ are given by
(see for example \cite{MarinThesis})
\begin{eqnarray}\label{deltaOm}\nonumber
\delta\xi_{k}&=&\frac{1}{h}\left(\Ad_{w_{k}}\eta_{k+1}-\eta_{k}+\frac{h}{2}\ad_{\xi_{k}}\eta_{k}-\frac{h}{2}\ad_{\xi_{k}}(\Ad_{w_{k}}\eta_{k+1})+\frac{h^{2}}{4}\xi_{k}\eta_{k}\xi_{k}\right.\\
&&-\left.\frac{h^{2}}{4}\xi_{k}(\Ad_{w_{k}}\eta_{k+1})\xi_{k}\right).
\end{eqnarray}

\subsubsection{Example: optimal control of a heavy top with
two internal rotors}\label{rotors}

The following example illustrates the study of underactuated mechanical control systems on Lie algebroids and the construction of variational integrators for such systems. It is the optimal control problem of the
upright spinning of the heavy top (see \cite{Chang} and reference
therein).

\begin{center}
\begin{figure}[h!]
\includegraphics[scale=.45]{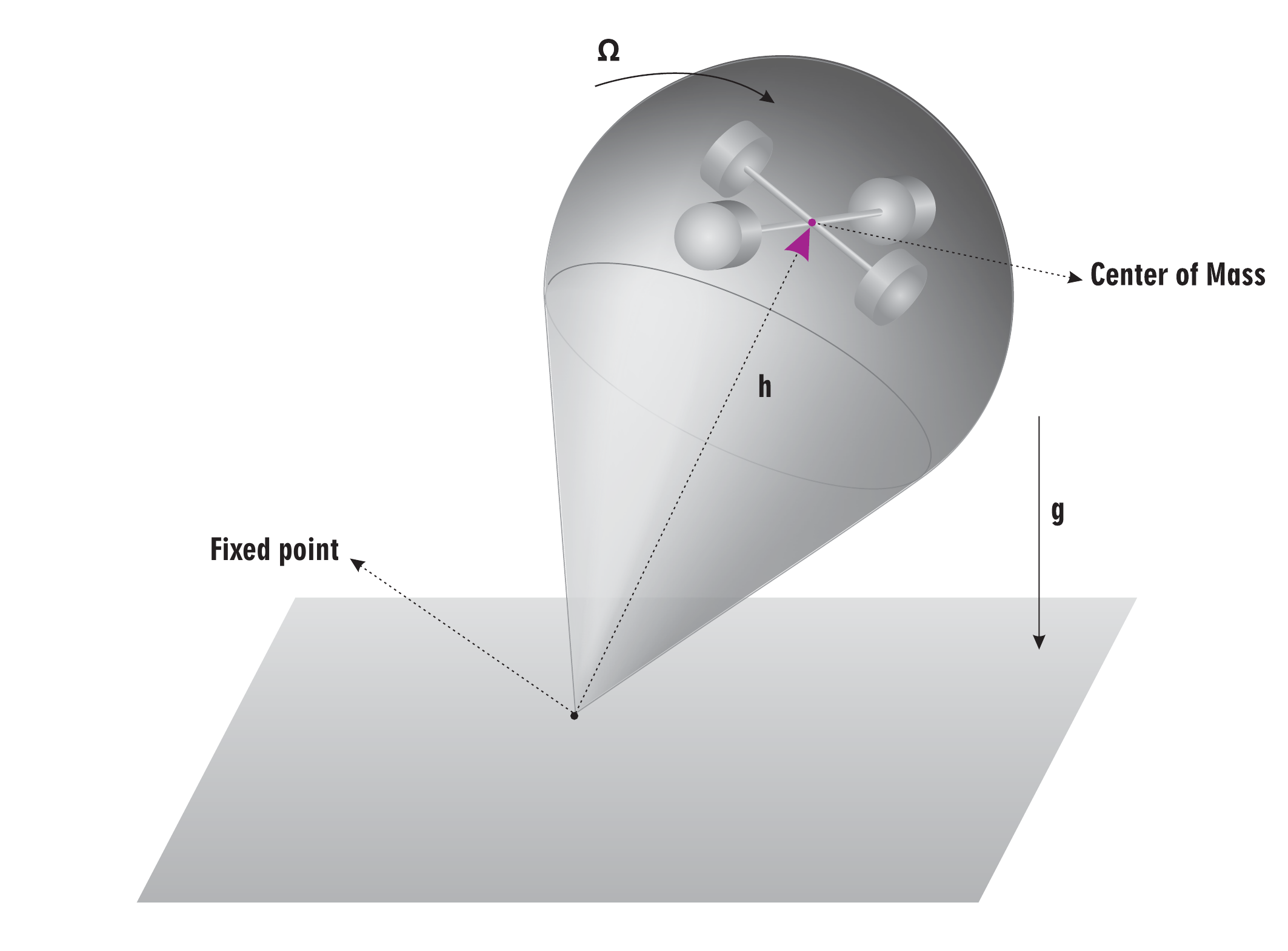} 
\caption{Heavy top with two rotors, each consisting of
two rigidly coupled disks.}\label{figureheavytop}
\end{figure}
\end{center}

Consider the top with two rotors so that each rotor's rotation axis is parallel to
the first and the second principal axes of the top as in Figure \ref{figureheavytop}. Let
$I_1,I_2,I_3$ be the moments of inertia of the top in the body fixed
frame. Let $J_1,J_2$ be the moments of inertia of the rotors around
their rotation axes and $J_{j1},J_{j2},J_{j3}$ be the moments of
inertia of the $j$-th rotor, with $j=1,2$, around the first, the
second and the third principal axes, respectively. Also we define
the quantities $\bar{I}_1=I_1+J_{11}+J_{21},$
$\bar{I}_2=I_2+J_{12}+J_{22},$ $\bar{I}_3=I_3+J_{13}+J_{23},$
$\gamma_{1}=\bar{I}_{1}+J_{1}$ and $\gamma_{2}=\bar{I}_{2}+J_{2}.$

Let $M$ be the total mass of the system, $g$ the magnitude of the
gravitational acceleration and $h$ the distance from the origin to
the center of mass of the system.

The system is modeled on the transformation Lie algebroid
$A=\mathbb{R}^{3}\times\mathfrak{so}(3)\times
T(\mathbb{S}^{1}\times\mathbb{S}^{1})$ over the manifold 
$Q=\mathbb{R}^{3}\times(\mathbb{S}^{1}\times\mathbb{S}^{1})$, where the
anchor map $\rho:\mathbb{R}^{3}\times\mathfrak{so}(3)\times
T(\mathbb{S}^{1}\times\mathbb{S}^{1})\to
T(\mathbb{R}^{3}\times\mathbb{S}^{1}\times\mathbb{S}^{1})$ is locally given
by
$$\rho(\Gamma,\Omega,\theta_1,\theta_2,\dot{\theta}_1,\dot{\theta}_{2})=(\Gamma,\theta_1,\theta_2,\Gamma\times\Omega,\dot{\theta}_1,\dot{\theta}_2).$$
Here
$\Omega=(\Omega_1,\Omega_2,\Omega_3)\in\mathfrak{so}(3)\simeq\R^{3}$
is the angular velocity of the top in the body fixed frame,
$\Gamma=(\Gamma_1,\Gamma_2,\Gamma_3)\in\mathbb{R}^{3}$ represents the unit vector
with the direction opposite to the gravity as seen from the body and
$\theta=(\theta_1,\theta_2)$ is the rotation angle of rotors around
their axes. 

If we denote by $E_{i}$ $(i=1,2,3),$ the standard basis of matrices of
$\mathfrak{so}(3),$ given by $$E_1=\left(
                        \begin{array}{ccc}
                          0 & 0 & 0 \\
                          0 & 0 & -1 \\
                          0 & 1 & 0 \\
                        \end{array}
                      \right),\quad E_2=\left(
                                      \begin{array}{ccc}
                                        0 & 0 & 1 \\
                                        0 & 0 & 0 \\
                                        -1 & 0 & 0 \\
                                      \end{array}
                                    \right),\quad E_3=\left(
                                                        \begin{array}{ccc}
                                                          0 & -1 & 0 \\
                                                          1 & 0 & 0 \\
                                                          0 & 0 & 0 \\
                                                        \end{array}
                                                      \right)
$$ then the basis of sections of $A$ is given
by the elements 
\begin{align*}
X^{E_i}(\Gamma,\theta_1,\theta_2)&=\left(\Gamma, E_i,\theta_1,\theta_2,0,0\right),\\
X^{\theta_{1}}(\Gamma,\theta_1,\theta_2)&=\left(\Gamma,0,\theta_1,\theta_2,1,0\right),\\
X^{\theta_{1}}(\Gamma,\theta_1,\theta_2)&=\left(\Gamma,0,\theta_1,\theta_2,0,1\right)\end{align*}
with $i=1,2,3$.

The Lie bracket of sections of $A$ is determined by
\begin{align*}\llbracket X^{E_1},X^{E_2}\rrbracket&=X^{[E_1,E_2]}=X^{E_3},\\
\llbracket X^{E_1},X^{E_3}\rrbracket&=X^{[E_1,E_3]}=-X^{E_2},\\
\llbracket X^{E_2},X^{E_3}\rrbracket&= X^{[E_2,E_3]}=X^{E_1},\\
\llbracket X^{\theta_r}, X^{\theta_s}\rrbracket&=0,\end{align*}
with $r,s=1,2$ and
$\llbracket X^{\theta_s},X^{E_i}\rrbracket=0,$
for $s=1,2$ and $i=1,2,3$.

The reduced Lagrangian
$l:\mathbb{R}^{3}\times\mathfrak{so}(3)\times
T(\mathbb{S}^{1}\times\mathbb{S}^{1})\to\R$ is given by

$$l(\Gamma,\Omega,\dot{\theta})=\frac{1}{2}\left(
                                                \begin{array}{c}
                                                  \Omega_1 \\
                                                  \Omega_2 \\
                                                  \Omega_3 \\
                                                  \dot{\theta}_1 \\
                                                  \dot{\theta}_2 \\
                                                \end{array}
                                              \right)^{T}\left(
                                                           \begin{array}{ccccc}
                                                             \gamma_1 & 0 & 0 & J_1 & 0 \\
                                                             0 & \gamma_2 & 0 & 0 & J_2 \\
                                                             0 & 0 & \bar{I}_3 & 0 & 0 \\
                                                             J_1 & 0 & 0 & J_1 & 0 \\
                                                             0 & J_2 & 0 & 0 & J_2 \\
                                                           \end{array}
                                                         \right)\left(
                                                                  \begin{array}{c}
                                                                    \Omega_1 \\
                                                                    \Omega_2 \\
                                                                    \Omega_3 \\
                                                                    \dot{\theta}_1 \\
                                                                    \dot{\theta}_2 \\
                                                                  \end{array}
                                                                \right)-Mgh\Gamma_{3}.$$

The Euler-Lagrange equations for $l$ are given by
\begin{align*}
\frac{d}{dt}\left(\frac{\partial l}{\partial\Omega}\right)&=\frac{\partial l}{\partial\Omega}\times\Omega+Mgh\Gamma\times e_3,\\
\frac{d}{dt}\left(\frac{\partial l}{\partial\dot{\theta}_{i}}\right)&=0,\quad i=1,2;
\end{align*} together with the admissibility condition $\dot{\Gamma}=\Gamma\times\Omega.$

Next, we add controls in our picture. Each rotor can
be controlled is such way the controlled Euler-Lagrange equations are now
\begin{align*}
\frac{d}{dt}\left(\frac{\partial l}{\partial\Omega}\right)=&\frac{\partial l}{\partial\Omega}\times\Omega+Mgh\Gamma\times e_3,\\
\frac{d}{dt}\left(\frac{\partial l}{\partial\dot{\theta}_{i}}\right)=&u_i,\quad i=1,2;\\
\dot{\Gamma}=&\Gamma\times\Omega.
\end{align*}
where $e_3=(0,0,1)$.
That is, \begin{align*}
\gamma_1\dot{\Omega}_1+J_1\ddot{\theta}_1-\gamma_2\Omega_2\Omega_3+\dot{\Omega}_3\bar{I}_3\Omega_2=&Mgh\Gamma_2,\\
\gamma_2\dot{\Omega}_2+J_2\ddot{\theta}_2+\gamma_1\Omega_1\Omega_3-J_1\dot{\theta}_1\Omega_3=&-Mgh\Gamma_1,\\
\bar{I}_3\dot{\Omega}_3-\gamma_1\Omega_1\Omega_2-J_1\dot{\theta}_1\Omega_2+\gamma_2\Omega_2\Omega_1+J_2\dot{\theta}_2\Omega_1=&0,\\
J_{1}(\dot{\Omega}_1+\ddot{\theta}_1)=&u_1,\\
J_2(\dot{\Omega}_2+\ddot{\theta}_2)=&u_2,
\end{align*} together with the admissibility conditions $$
\dot{\Gamma}_1=\Gamma_2\Omega_3-\Gamma_3\Omega_2,\quad
\dot{\Gamma}_2=\Gamma_3\Omega_1-\Gamma_1\Omega_3,\quad
\dot{\Gamma}_3=\Gamma_1\Omega_2-\Gamma_2\Omega_1
$$ where $\Gamma=(\Gamma_1,\Gamma_2,\Gamma_3)\in\mathbb{R}^{3}.$

The optimal control problem consists on finding an admissible curve
$\gamma(t)=(\Gamma(t),\Omega(t),\theta(t),u_{i}(t))$ of the state
variables and control inputs, satisfying the controlled equations given above, the boundary conditions and minimizing the cost functional
$$\mathcal{J}=\frac{1}{2}\int_{0}^{T}\left(u_{1}^{2}+u_{2}^{2}\right) dt.$$ 

This optimal control problem is equivalent to solve the
second-order variational problem determined by 
$$\min_{(\Omega(\cdot),\theta(\cdot),u(\cdot))}\int_{0}^{T}\widetilde{L}\left(\Omega,\theta,\dot{\Omega},\dot{\theta},\ddot{\theta}\right) dt,$$
and subjected to the second-order constraints
$\Phi^{A}:\mathbb{R}^{3}\times\mathbb{R}^{3}\times 2\mathfrak{so}(3)\times
T^{(2)}(\mathbb{S}^{1}\times\mathbb{S}^{1})\to\R,$ $A=1,\ldots,3$;
\begin{align*}
\Phi^{1}\left(\Omega,\theta,\Gamma,\dot{\Omega},\dot{\theta},\dot{\Gamma},\ddot{\theta}\right)=&\gamma_1\dot{\Omega}_1+J_1\ddot{\theta}_1-\gamma_2\Omega_2\Omega_3+\dot{\Omega}_3\bar{I}_3\Omega_2-Mgh\Gamma_2=0,\\
\Phi^{2}\left(\Omega,\theta,\Gamma,\dot{\Omega},\dot{\theta},\dot{\Gamma},\ddot{\theta}\right)=&\gamma_2\dot{\Omega}_2+J_2\ddot{\theta}_2+\gamma_1\Omega_1\Omega_3-J_1\dot{\theta}_1\Omega_3+Mgh\Gamma_1=0,\\
\Phi^{3}\left(\Omega,\theta,\Gamma,\dot{\Omega},\dot{\theta},\dot{\Gamma},\ddot{\theta}\right)=&\bar{I}_3\dot{\Omega}_3-\gamma_1\Omega_1\Omega_2-J_1\dot{\theta}_1\Omega_2+\gamma_2\Omega_2\Omega_1+J_2\dot{\theta}_2\Omega_1=0,
\end{align*} together with the admissibility condition $\dot{\Gamma}-\Gamma\times\Omega=0$ and where $\widetilde{L}:
\mathfrak{so}(3)\times\mathfrak{so}(3)\times
T^{(2)}(\mathbb{S}^{1}\times\mathbb{S}^{1})\to\R$ is given by
\begin{equation*}
\widetilde{L}\left(\Omega,\theta,\dot{\Omega},\dot{\theta},\ddot{\theta}\right)=\mathcal{C}\left(\Omega,\theta,\dot{\Omega},\dot{\theta},\frac{d}{dt}\left(\frac{\partial l}{\partial\dot{\theta}}\right)\right)
=\frac{J_{1}^{2}}{2}(\dot{\Omega}_1+\ddot{\theta}_1)^{2}+\frac{J_2^{2}}{2}(\dot{\Omega}_2+\ddot{\theta}_2)^{2}
\end{equation*} where $\mathcal{C}=\frac{1}{2}(u_1^2+u_2^2)$.

Therefore, the optimality conditionsare determined by the constrained second-order Euler-Lagrange equations given by 

\begin{align*}
0=&J_{1}(\dddot{\Omega}_{1}+\theta_{1}^{(4)})+\ddot{\lambda}_{1}-\lambda_{2}\dot{\Omega}_{3}+\dot{\lambda}_{2}\Omega_{3}-\lambda_{3}\dot{\Omega}_{2},\\
0=&J_{2}(\dddot{\Omega}_{2}+\theta_{2}^{(4)})+\ddot{\lambda}_{2}-\lambda_{3}\dot{\Omega}_{1}-\dot{\lambda}_{2}\Omega_{1},\\
0=&\gamma_1\dot{\Omega}_1+J_1\ddot{\theta}_1-\gamma_2\Omega_2\Omega_3+\dot{\Omega}_3\bar{I}_3\Omega_2-Mgh\Gamma_2,\\
0=&\gamma_2\dot{\Omega}_2+J_2\ddot{\theta}_2+\gamma_1\Omega_1\Omega_3-J_1\dot{\theta}_1\Omega_3+Mgh\Gamma_1,\\
0=&\bar{I}_3\dot{\Omega}_3-\gamma_1\Omega_1\Omega_2-J_1\dot{\theta}_1\Omega_2+\gamma_2\Omega_2\Omega_1+J_2\dot{\theta}_2\Omega_1,\\
0=&A(\lambda_3B+\dot{\lambda}_3)+\lambda_3\dot{A}+J_1^{2}(\dddot{\Omega}_{1}+\theta_{1}^{(4)}-\ddot{\Omega}_1B+\dddot{\theta}_{1}B)\\
&+\lambda_2\gamma_1(\dot{\Omega}_3+\Omega_1\Omega_3+\Omega_2\Omega_3+\Omega_3^2)+\gamma_1(\dot{\lambda}\Omega_3-\ddot{\lambda}_1-\dot{\lambda}_1B),\\
0=&\dot{\lambda}_1C+(\dot{\lambda}_{3}+\lambda B)(\gamma_2\Omega_1-\gamma_1\Omega_1-J_1\dot{\theta}_1)-\ddot{\lambda}_2\gamma_2+\lambda_1\dot{C}\\
&-\lambda_3(\gamma_2\dot{\Omega}_1-\gamma_1\dot{\Omega}_1-J_{1}\ddot{\theta}_1)-J_2^2(\dddot{\Omega}_2+\theta_2^{(4)}+B)+B(\lambda_1C-\gamma_2\dot{\lambda}_2)\\
0=&(\dot{\lambda}_2+B\lambda_2)(\gamma_1\Omega_1-J_1\dot{\theta_1})-B\lambda_1(\gamma_2\Omega_2+\bar{I}_3\dot{\Omega}_2)-B\bar{I}_3(\dot{\lambda}_1\Omega_2+\dot{\lambda}_3)\\
&+\lambda_2(\gamma_1\dot{\Omega}_1-J_1\ddot{\theta}_1)-\gamma_2(\dot{\lambda}_1\Omega_2+\lambda_1\dot{\Omega}_2)-\bar{I}_3(\ddot{\lambda}_1\Omega_2-\ddot{\lambda}_3-2\dot{\lambda}_1\dot{\Omega}_2-\lambda_1\ddot{\Omega}_2\\
0=&\dot{\Gamma}-\Gamma\times\Omega
\end{align*} where $A=\gamma_2\Omega_2-\gamma_1\Omega_2+J_2\dot{\theta}_2$, $B=\Omega_1+\Omega_2+\Omega_3$ and $C=\dot{\Omega}_3\bar{I}_3-\gamma_2\Omega_3$.

As before, we use the Cayley transformation on $SO(3)$ to describe the
discretization of the optimal control problem for the heavy top with internal
rotors. We redefine the Lagrangian $\widetilde{L}$ and the constraints
$\Phi^{A}$ as $\mathcal{L}: 2\mathfrak{so}(3)\times
T^{(2)}(\mathbb{S}^{1}\times\mathbb{S}^{1})\to\R$ and
$\widehat{\Phi}^{A}:T\mathbb{R}^{3}\times
2\mathfrak{so}(3)\times
T^{(2)}(\mathbb{S}^{1}\times\mathbb{S}^{1})\to\R$ by
$$\mathcal{L}(\xi_k,\theta,\xi_{k+1},\dot{\theta},\ddot{\theta}):=\widetilde{L}\left(\xi_{k+1/2},\theta,\xi_{k,k+1},\dot{\theta},\ddot{\theta}\right)$$
and
$$\widehat{\Phi}^{A}(\Gamma,\xi_k,\theta,\dot{\Gamma},\xi_{k+1},\dot{\theta},\ddot{\theta}):=\Phi^{\alpha}\left(\Gamma,\xi_{k+1/2},\theta,\dot{\Gamma},\xi_{k,k+1},\dot{\theta},\ddot{\theta}\right),$$
where $\xi_k,\xi_{k+1}\in\mathfrak{so}(3)$ and $h>0$ is a fixed real
number with $T=kh$, $k=0,\ldots,N$.

The discrete second-order Lagrangian
$L_{d}:2\mathfrak{so}(3)\times
3(\mathbb{S}^{1}\times\mathbb{S}^{1})\to\R$ associated with $\mathcal{L}$ is given by \begin{equation*}
L_{d}(\xi_k,\theta_k^{i},\xi_{k+1},\theta_{k+1}^{i},\theta_{k+2}^i):=
h\mathcal{L}\left(\xi_{k+1/2},\theta_{k+2/3}^{i},\xi_{k,k+1},\theta_{k,k+2}^{i},\theta_{k+2/h^2}^{i}\right)
\end{equation*} for $i=1,2$ and the discrete constraints $\Phi_{d}^{A}:2\mathfrak{so}(3)\times
3(\mathbb{S}^{1}\times\mathbb{S}^{1})\to\R$ associated with $\widehat{\Phi}^{A}$ by
\begin{align*}
&\Phi_{d}^{A}(\Gamma_k,\xi_k,\theta_k^{i},\xi_{k+1},\Gamma_{k+1},\theta_{k+1}^{i},\theta_{k+2}^i):=\\
&h\widehat{\Phi}^{A}\left(\Gamma_k,\xi_{k+1/2},\theta_{k+2/3}^{i},\Gamma_{k,k+1},\xi_{k,k+1},\theta_{k,k+2}^{i},\theta_{k+2/h^2}^{i}\right)
\end{align*}
for $i=1,2$ , where
$$\theta_{k+2/3}^{i}=\frac{\theta_{k}^{i}+\theta_{k+1}^{i}+\theta_{k+2}^{i}}{3},\quad
\theta_{k,k+2}^{i}=\frac{\theta_{k+2}^{i}-\theta_k^{i}}{2h},\quad
\theta_{k+2/h^2}^i=\frac{\theta_{k+2}^{i}-2\theta_{k+1}^{i}+\theta_k^{i}}{h^2},$$
and $h\xi_k=\hbox{cay}^{-1}(\omega_k)\in\mathfrak{so}(3)$ for $\omega_k\in SO(3)$. Here
$$\xi_k=\left(
                                                                                                                                \begin{array}{ccc}
                                                                                                                                  0 & -(\xi_3)_k & (\xi_2)_k \\
                                                                                                                                  (\xi_3)_k & 0 & -(\xi_1)_k \\
                                                                                                                                  -(\xi_2)_k & (\xi_1)_k & 0 \\
                                                                                                                                \end{array}
                                                                                                                            \right)\in \mathfrak{so}(3).
$$

The geometric integrator is given by extremizing the discrete cost function defined by 
$$C_{d}=\sum_{k=0}^{N-1}\left[L_{d}(\xi_k,\theta_{k}^{i},\xi_{k+1},\theta_{k+1}^{i},\theta_{k+2}^{i})+(\lambda_{A}^{k})^{T}\Phi_{d}^{A}(\Gamma_{k},\xi_{k},\theta_{k}^{i},\Gamma_{k+1},\xi_{k+1},\theta_{k+1}^{i},\theta_{k+2}^{i})\right]$$
where $\lambda_{A}\in\R^{3}$, $A=1,\ldots,3$ are Lagrange multipliers. That is, it is given by the solutions of the discrete second-order constrained Euler-Lagrange equations associated to the discrete extended Lagrangian $\widehat{L}_d=L_d+\lambda_{A}\Phi^{A}_{d}:2\mathfrak{so}(3)\times
3(\mathbb{S}^{1}\times\mathbb{S}^{1})\times\R^{3}\to\R$ where  
\begin{align*}
L_d=&\frac{J_{1}^{2}}{2}\left((\xi_{k,k+1})_1+\theta_{k+2/h^{2}}^{1}\right)^{2}+\frac{J_{2}^{2}}{2}\left(\frac{(\xi_{k+1})_2-(\xi_k)_2}{h}+\frac{\theta_{k+2}^{2}-2\theta_{k+1}^{2}+\theta_k^{2}}{h^2}\right)^{2},\\
\Phi_{d}^{1}=&\gamma_{1}\left(\frac{(\xi_{k+1})_1-(\xi_k)_1}{h}\right)+J_{1}\left(\frac{\theta_{k+2}^{1}-2\theta_{k+1}^{1}+\theta_k^{1}}{h^{2}}\right)-MghY_{k}\\
&-\gamma_{2}\left(\frac{((\xi_k)_2+(\xi_{k+1})_2)((\xi_k)_3+(\xi_{k+1})_{3})}{4}\right)\\
&+\bar{I}_{3}\left(\frac{((\xi_{k+1})_{3}-(\xi_k)_{3})((\xi_{k+1})_2+(\xi_k)_2)}{2h}\right),\\
\Phi_{d}^{2}=&\gamma_{2}\left(\frac{(\xi_{k+1})_2-(\xi_k)_2}{h}\right)+J_{2}\left(\frac{\theta_{k+2}^{2}-2\theta_{k+1}^{2}+\theta_k^{2}}{h^{2}}\right)\\
&+\gamma_{1}\left(\frac{((\xi_k)_1+(\xi_{k+1})_1)((\xi_k)_3+(\xi_{k+1})_{3})}{4}\right)\\
&-J_1\left(\frac{(\theta_{k+2}^{1}-\theta_k^{1})((\xi_{k+1})_3+(\xi_k)_3)}{2h}\right)-MghY_{k},\end{align*}
\begin{align*}\Phi_{d}^{3}=&\bar{I}_{3}\left(\frac{(\xi_{k+1})_3-(\xi_k)_3}{h}\right)-\gamma_{1}\left(\frac{((\xi_{k+1})_{1}+(\xi_k)_1)((\xi_{k+1})_{2}+(\xi_{k})_2)}{4}\right)\\
&+\gamma_{2}\left(\frac{((\xi_{k+1})_2+(\xi_{k})_2)((\xi_{k+1})_1+(\xi_{k})_1)}{4}\right)\\&+J_{2}\left(\frac{(\theta_{k+2}^{2}-\theta_{k}^{2})((\xi_k)_1+(\xi_{k+1})_1)}{4h}\right)\\
&-J_{1}\left(\frac{(\theta_{k+2}^{1}-\theta_k^{1})((\xi_{k})_{2}+(\xi_{k+1})_2)}{4h}\right),
\end{align*} together with $\displaystyle{\Gamma_{k+1}=\Gamma_k\left(\frac{cay^{-1}(\omega_k)+cay^{-1}(\omega_{k+1})}{2h}\right)}$.


\section{Lagrangian submanifolds generating discrete
dynamics}\label{lagsubmanifoldsgroupoids}

In this section we study how a Lagrangian submanifold of a particular cotangent groupoid can be used to give a more geometric and intrinsic point of view  
of discrete
second-order dynamics. 
Moreover, we will study the preservation properties of the derived discrete implicit dynamics. We also study discrete second-order constrained systems. Particularly, from this geometrical framework, we
will analyze some geometric properties of the associated discrete flow. Finally, we will study the theory of reduction under  symmetries.

\subsection{The prolongation of a Lie groupoid over a fibration}\label{subsection prolongation}

Given a Lie groupoid $ G \rightrightarrows Q $ with structural maps $\alpha:G\to Q$, $\beta:G\to Q$, $\epsilon:Q\to G$, $i:G\to G$, $m:G_2\to G$, and a fibration
$\pi:P\to Q,$  we consider the set
$$\ppig=P_{\pi}\times_{\alpha}G_{\beta}\times_{\pi}P=\{(p,g,p')\in
P\times G\times P/\pi(p)=\alpha(g),\beta(g)=\pi(p')\}.$$ $\ppig$ has a Lie groupoid structure over $P,$ where the structural maps are
given by \begin{align*}
\alpha^{\pi}&:\ppig\longrightarrow P,\quad (p,g,p')\longmapsto p;\\
\beta^{\pi}&:\ppig\longrightarrow P,\quad (p,g,p')\longmapsto p';\\
m^{\pi}&:(\ppig)_2\longrightarrow\ppig, \quad ((p,g,p'),(p',h,p''))\longmapsto (p,gh,p'');\\
\epsilon^{\pi}&:P\longrightarrow\ppig,\quad p\longmapsto(p,\epsilon(\alpha(p)),p);\\
 i^{\pi}&:\ppig\longrightarrow\ppig, \quad (p,g,p')\longmapsto
(p',g^{-1},p).
\end{align*} $\ppig$ is called \textit{prolongation of $G$ over $\pi:P\to Q$} (See
\cite{Mack},\cite{MMM} and \cite{Sau}).

Next, we consider the prolongation $\palg$ of the Lie
groupoid $G$ over its source map $\alpha:G\to Q$, that is, one can
consider the subset of $3G:=G\times G\times G$,
$$\palg=G_{\alpha}\times_{\alpha}G_{\beta}\times_{\alpha}G=\{(g,h,r)\in
3G/\alpha(g)=\alpha(h),\beta(h)=\alpha(r)\}.$$ $\palg$ is a Lie
groupoid over $G.$ Moreover, $G_2\subseteq\palg$ where the
inclusion is given by
\begin{eqnarray*}
i_{G_2}:G_2&\hooklongrightarrow&\palg\\
(g,h)&\longmapsto&(g,g,h).
\end{eqnarray*}
Now, we construct the Lie
algebroid associated with $\palg.$ This will be identified with the
prolongation $\palag$ of $AG$ over $\alpha:G\to Q$, where $AG$ is the Lie algebroid associated with $G$ with bundle projection $\tau_{AG}:AG\to Q$.

\begin{definition}\label{defalgasoc}
The Lie algebroid associated with a prolongation of a Lie groupoid
$G$ over $\alpha$ is given by,
$$\algpalg=\{(a_{\epsilon(\alpha(g))};Y_g)\in A_{\alpha(g)}G\times T_g G\mid(T_{g}\alpha)(Y_g)=(T_{\epsilon(\alpha(g))}\beta)(a_{\epsilon(\alpha(g))})\}.$$ $\alpalg$ is a Lie algebroid over $\palg$ with bundle projection denoted by $\tau_{\alpalg}:\alpalg\to\palg$.
\end{definition}

\begin{remark}

The prolongation of a Lie algebroid over a fibration has a Lie algebroid structure (See Appendix B). If we consider the linear isomorphism
\begin{eqnarray*}
(\Psi^{\alpha})_g:\algpalg&\longrightarrow&\pgalag\subset A_{\alpha(g)}G\times
T_g G\\
(0_g,a_{\epsilon(\alpha(g))},Y_g) &\longmapsto&
(a_{\epsilon(\alpha(g))},Y_g)\; ,\quad\forall g\in G
\end{eqnarray*}then the mapping $(\Psi^{\alpha})_g,$ $g\in G$
induces an isomorphism $\Psi^{\alpha}:\alpalg\to\palag$ between the
Lie algebroids $\alpalg$ and $\palag$ (see
\cite{HiMa} and \cite{MMM} for more details about the construction of this linear isomorphism).\hfill$\diamond$
\end{remark}

From \eqref{defalgasoc} a section $Z$ of $\alpalg$ can be expressed as
$$Z(g)=(X(\alpha(g)),Y(g))$$ where $g\in G$, $X\in\Gamma(\tau_{AG})$ and
$Y$ is a vector field on $G$ such that $T\beta(X)=T\alpha(Y).$

The corresponding left-invariant and right-invariant vector fields associated with the section $Z\in\Gamma(\tau_{\alpalg})$
are
\begin{align}
\overleftarrow{Z}(g,h,r)=&(0,\overleftarrow{X}(h),Y(r)),\label{zl}\\
\overrightarrow{Z}(g,h,r)=&(-Y(g),\overrightarrow{X}(h),0),\label{zr}
\end{align} with $(g,h,r)\in\palg$.

Given a basis of sections of $AG$ one can obtain a basis of
sections of $\alpalg$, denoted by $\{Z_1,Z_2\}$, with
\begin{eqnarray}\label{z1z2} Z_1=(-X,\overrightarrow{X})\hbox{ and }
Z_2=(0,\overleftarrow{X})\end{eqnarray} where $X\in\Gamma(\tau_{AG}),
\overrightarrow{X}\in\overrightarrow{\mathfrak{X}}(G)$ and
$\overleftarrow{X}\in\overleftarrow{\mathfrak{X}}(G)$. Here, we are  using  the
notation $\overrightarrow{\mathfrak{X}}(G)$ (resp.,
$\overleftarrow{\mathfrak{X}}(G)$) for the set of right-invariant
(resp., left-invariant) vector fields on $G$. 





The next result is a direct application of the construction given before and it will be useful when we derive the discrete second-order dynamics for a second-order discrete Lagrangian.
 
\begin{lemma}\label{defsection}
Let $\{Z_1,Z_2\}$ be a basis of sections of  $A(\palg)$,
 where $Z_1=(-X,\overrightarrow{X})$ and
$Z_2=(0,\overleftarrow{X})$, $X\in\Gamma(\tau_{AG})$. For $(g,h,r)\in\palg$ the
associated left and right invariant vector fields for $Z_1$ and $Z_2$ are given by 
\begin{eqnarray*}
\overrightarrow{Z_1}(g,h,r)=
(-\overrightarrow{X}(g),-\overrightarrow{X}(h),0)\qquad&&
\overrightarrow{Z_2}(g,h,r)=
(-\overleftarrow{X}(g),0,0)\\
\overleftarrow{Z_1}(g,h,r)=(0,-\overleftarrow{X}(h),\overrightarrow{X}(r))\qquad&&
\overleftarrow{Z_2}(g,h,r)=(0, 0, \overleftarrow{X}(r)).
\end{eqnarray*}

\end{lemma}

\subsection{Generating Lagrangian submanifolds and dynamics on Lie groupoids}

Let $G\rightrightarrows Q$ be a Lie groupoid with source and target
map $\alpha,\beta:G\to Q$ respectively, and we consider the
prolongation of $G$ over its source map, $\palg$. We denote by $\alpha^{\alpha},\beta^{\alpha}:\palg\to
G$ the source and target maps of this Lie groupoid. Let
$\tau_{A^{*}(\palg)}:A^{*}(\palg)\rightarrow G$ be the dual of the
vector bundle associated with the Lie algebroid
$\tau_{\alpalg}:\alpalg\rightarrow G$. The Lie groupoid (cotangent groupoid)
$T^{*}(\palg)\rightrightarrows A^{*}(\palg)$ is a symplectic
groupoid (see example 3 in section \ref{examples}).

In what follows, we show how the discrete dynamics associated with a discrete second-order Lagrangian $L_d:G_2\to\R$ is generated by a Lagrangian submanifold of the cotangent groupoid $T^{*}(\palg)$. 

Remember that  given a manifold $Q$ and a function $S:Q\to\R$, the submanifold ${\rm Im }\;  \hbox{d}S\subset T^{*}Q$ is Lagrangian. There is a more general construction given to \'Sniatycki and Tulczyjew \cite{ST} (see also \cite{Tulczy1} and \cite{Tulczy2}) which we will use to generate the discrete dynamics.

\begin{theorem}[\'Sniatycki and Tulczyjew \cite{ST}]\label{tulczyjew}
Let $Q$ be a smooth manifold, $ N \subset Q $ a submanifold, and $ S
  \colon N \rightarrow {\mathbb R}$.  Then
    \begin{eqnarray*}&&\Sigma _S = \bigl\{ \mu \in T ^\ast Q \mid \pi _Q (\mu) \in N \text{
        and } \left\langle \mu, v \right\rangle = \left\langle
        \mathrm{d} S , v \right\rangle\\
&&\qquad\qquad\qquad      \hbox{ for all } v \in T N \subset T Q \hbox{ such that } \tau
      _Q (v) = \pi _Q (\mu) \bigr\}\end{eqnarray*}
  is a Lagrangian submanifold of $ T ^\ast Q $. Here $\pi_Q:T^{*}Q\to Q$ and $\tau_Q:TQ\to Q$ denote the cotangent and tangent bundle projections, respectively.

\end{theorem} 


Turning back to the groupoid formulation, immediately from Theorem \eqref{tulczyjew}, the discrete second-order Lagrangian $L_d:G_2\to\mathbb{R}$
generates a Lagrangian submanifold $\Sigma_{L_d}\subset T^{*}(\palg)$ of
the symplectic Lie groupoid $(T^{*}(\palg),\omega_{\palg})$ where
$\omega_{\palg}$ denotes the canonical symplectic 2-form on
$T^{*}(\palg)$. Denoting by $i_{G_2}:G_2\hookrightarrow\palg$ the
inclusion defined by $i_{G_2}(g_1,g_2)=(g_1,g_1,g_2)$, we have 
$$\Sigma_{L_d}=\{\mu\in T^{*}(\palg)\mid i^{*}_{G_2}\mu=dL_d\}\subset
T^{*}(\palg)$$ is a Lagrangian submanifold of
$(T^{*}(\palg),\omega_{\palg}).$

The relationship among these spaces is summarized in the following
diagram

\begin{equation*}
  \xymatrix{
  &\Sigma _{L_d} \mathrel{\ar@{^(->}[r]} & T
  ^\ast(\palg)
    \ar@<.5ex>[r] ^{ \tilde{ \alpha } }  \ar@<-.5ex>[r]_{ \tilde{
        \beta } }  \ar[d]& A ^\ast (\palg) \ar[d]\\ \R &   \ar[l]_{L_d} G_2
    \mathrel{\ar@{^(->}[r]^{i_{G_2}}}& \palg \ar@<.5ex>[r] ^{\alpha^{\alpha}} \ar@<-.5ex>[r]
    _{\beta^{\alpha}}  & G  }
\end{equation*} where from now on we will denote $\tilde{\alpha}$ and $\tilde{\beta}$ the
source and target maps, respectively, of the Lie groupoid
$T^{*}(\palg)\rightrightarrows A^{*}(\palg)$.

Given an element $\mu\in T^{*}_{(g,h,r)}(\palg)\hbox{ with }(g,h,r)\in\palg$
the source and target maps of $T^{*}(\palg)$ are defined such that, for all section
$Z\in\Gamma(\tau_{A(\palg)})$,  \begin{align}
\bra \tilde{\alpha}(\mu),Z(\alpha(g))\ket=&\bra\mu,\overrightarrow{Z}(g,h,r)\ket \label{Zsections}\\
\bra\tilde{\beta}(\mu), Z(\beta(g))\ket=&\bra
\mu,\overleftarrow{Z}(g,h,r)\ket,\label{Zsections2}
\end{align} where $\overleftarrow{Z}$ and $\overrightarrow{Z}$ are
the corresponding left and right invariant vector fields associated
with the section $Z$ of $A(\palg)$ according to \eqref{zl} and
\eqref{zr}.

Denoting by
$$\gamma_{(g_k,g_{k+1})}=(\mu_{g_{k}},\tilde{\mu}_{g_{k}},\bar{\mu}_{g_{k+1}})\in T^{*}_{i_{G_2}(g_k,g_{k+1})}(\palg),$$ with $(g_k,g_{k+1})\in G_{2}$, the Lagrangian submanifold $\Sigma_{L_d}$ gives rise to the discrete second-order dynamics as we describe in the following.

\begin{definition}\label{defdynamics}
A sequence
$\gamma_{(g_{1},g_{2})},\ldots,\gamma_{(g_{N-1},g_{N})}\in
T^{*}(\palg)$ satisfy the second-order dynamics on
$\Sigma_{L_d}$ if
$\gamma_{(g_{1},g_{2})},\ldots,\gamma_{(g_{N-1},g_{N})}\in\Sigma_{L_d}$
and
$$\tilde{\alpha}(\gamma_{(g_k,g_{k+1})})=\tilde{\beta}(\gamma_{(g_{k-1},g_{k})})\hbox{ for } k=2,\ldots,N-1.$$
That is, $\gamma_{(g_{1},g_{2})},\ldots,\gamma_{(g_{N-1},g_{N})}$
are composable sequences on $T^{*}(\palg).$
\end{definition}


\begin{theorem}\label{teoremagroupoides1}
Let $G\rightrightarrows Q$ be a Lie groupoid and $L_d:G_2\to\R$ be a
discrete second-order Lagrangian.  Consider the 
Lagrangian submanifold $\Sigma_{L_d}$ of  the cotangent groupoid $T^{*}(\palg)$ generated by $L_d$. A
sequence $\gamma_{(g_{1},g_{2})},\ldots,\gamma_{(g_{N-1},g_{N})}$
satisfies the discrete second-order dynamics on $\Sigma_{L_d}$ if and only
if $\gamma_{(g_{1},g_{2})},\ldots,\gamma_{(g_{N-1},g_{N})}$ satisfy
\begin{align*}
\bra\overleftarrow{X}(g_k),\tilde{\mu}_{g_k}\ket+\bra Y(g_k),\bar{\mu}_{g_{k+1}}\ket=&\bra\overrightarrow{X}(g_{k+1}),\tilde{\mu}_{g_{k+1}}\ket-\bra Y(g_{k+1}),\mu_{g_{k+1}}\ket\\
\mu_{g_k}+\tilde{\mu}_{g_{k}}=&D_1L_d(g_k,g_{k+1})\\
\bar{\mu}_{g_{k+1}}=&D_2L_d(g_k,g_{k+1})
\end{align*} for  $k=1,\ldots,N-1$ and for any section $Z\in\Gamma(\tau_{\alpalg})$  according to \eqref{zl} and
\eqref{zr}.

\end{theorem}

\textit{Proof.}  Consider the  sequence $\gamma_{(g_{1},g_{2})},\ldots,\gamma_{(g_{N-1},g_{N})}$ in $T^{*}(\palg)$. Applying the definition of $\tilde{\alpha}$  \eqref{Zsections} and $\tilde{\beta}$  \eqref{Zsections2} to the relation $\tilde{\alpha}(\gamma_{(g_k,g_{k+1})})=\tilde{\beta}(\gamma_{(g_{k-1},g_{k})})$ for $k=2,\ldots,N-1$, we
have that for any section $Z\in\Gamma(\tau_{\alpalg})$ the sequence $\gamma_{(g_{1},g_{2})},\ldots,\gamma_{(g_{N-1},g_{N})}$ belongs to the Lagrangian submanifold $\Sigma_{L_d}$ if \begin{align}
\bra\overleftarrow{Z}(g_k,g_k,g_{k+1}),\gamma_{(g_{k},g_{k+1})}\ket=&\bra\overrightarrow{Z}(g_{k+1},g_{k+1},g_{k+2}),\gamma_{(g_{k+1},g_{k+2})}\ket
\label{dynamiceq0}\\
\mu_{g_k}+\tilde{\mu}_{g_{k}}=&D_1L_d(g_k,g_{k+1})\label{dynamiceq}\\
\bar{\mu}_{g_{k+1}}=&D_2L_d(g_k,g_{k+1}),\label{dynamiceq1}
\end{align} for $k=1,\ldots,N-1$. Using \eqref{zl} and \eqref{zr} the equations given above are equivalent to 
\begin{align}
\bra\overleftarrow{X}(g_k),\tilde{\mu}_{g_k}\ket+\bra Y(g_{k+1}),\bar{\mu}_{g_{k+1}}\ket=&\bra\overrightarrow{X}(g_{k+1}),\tilde{\mu}_{g_{k+1}}\ket-\bra Y(g_{k+1}),\mu_{g_{k+1}}\ket\label{dynamiceq2}\\
\mu_{g_k}+\tilde{\mu}_{g_{k}}=&D_1L_d(g_k,g_{k+1}\label{dynamiceq3})\\
\bar{\mu}_{g_{k+1}}=&D_2L_d(g_k,g_{k+1})\label{dynamiceq4}
\end{align} for
$k=1,\ldots,N-1$ as we claimed.
\hfill$\square$

\begin{remark}\label{remarkrelacion}
We have seen how the dynamics is only defined implicitly by a relation in
$T^{*}(\palg)$ rather that as an explicit discrete flow map. Therefore, the sequence 
$\gamma_{(g_{1},g_{2})},\ldots,\gamma_{(g_{N-1},g_{N})}\in
T^{*}(\palg)$ satisfy the discrete second-order dynamics on
$\Sigma_L$ if and only if each pair of successive elements
satisfies the relation
$$(\gamma_{(g_{k-1},g_k)},\gamma_{(g_{k},g_{k+1})})\in
(T^{*}(\palg))_{2}\cap(\Sigma_{L_d}\times\Sigma_{L_d}).$$\hfill$\diamond$
\end{remark}

Next, we show that the discrete dynamics described implicitly in Theorem \eqref{teoremagroupoides1} is equivalent to the discrete second-order Euler-Lagrange equations \eqref{delelg} given by the variational point of view.

\begin{theorem}\label{teoremagroupoides2}
Let $L_d:G_2\to\R$ be a discrete second order Lagrangian. For every
section $Z$ of $\Gamma(\tau_{A(\palg)})$ according to \eqref{zl} and
\eqref{zr} the discrete second order Euler-Lagrange
equations associated with $L_{d}$ are \begin{align*}0=&\ell^{*}_{g_{k}}\left(D_1L_d(g_{k},g_{k+1})+D_2 L_d(g_{k-1},g_{k})\right)\\&+(r_{g_{k+1}}\circ i)^{*}\left(D_1L_d(g_{k+1},g_{k+2})+D_2 L_d(g_k,g_{k+1})\right),\end{align*} for $k=2,\ldots,N-2$.
\end{theorem}

\textit{Proof.}  Let $Z$ be a section of $A(\palg)$ and consider the basis $\{Z_1,Z_2\}$ of section of $A(\palg)$ as in \eqref{z1z2}. Using  Lemma \eqref{defsection} in \eqref{dynamiceq0}  we get the relation
$$\bra\overleftarrow{Z_2}(g_k,g_k,g_{k+1}),\gamma_{(g_{k},g_{k+1})}\ket=\bra\overrightarrow{Z_2}(g_{k+1},g_{k+1},g_{k+2}),\gamma_{(g_{k+1},g_{k+2})}\ket$$ if and only if
 \begin{eqnarray}
 \bra\overleftarrow{X}(g_{k+1}),\bar{\mu}_{g_{k+1}}\ket&=&-\bra\overleftarrow{X}(g_{k+1}),\mu_{g_{k+1}}\ket\label{dynamiceq6}
\end{eqnarray}
for $k=1,\ldots,N-1$. Now, using the relations 
 \begin{align}
 \mu_{g_{k+1}}+\tilde{\mu}_{g_{k+1}}=&D_1L_d(g_{k+1},g_{k+2}),\label{dynamiceq7}\\
\bar{\mu}_{g_{k+1}}=&D_2L_d(g_k,g_{k+1}),
\label{dynamiceq8}
\end{align}we have, \begin{equation}\label{dynamiceq9}
 \bra\overleftarrow{X}(g_{k+1}),D_2L_d(g_k,g_{k+1})\ket=-\bra\overleftarrow{X}(g_{k+1}),D_1L_d(g_{k+1},g_{k+2})-\tilde{\mu}_{g_{k+1}}\ket.
\end{equation} Similarly, by Lemma \eqref{defsection} in \eqref{dynamiceq0},  we get the relation
$$\bra\overleftarrow{Z}_1(g_k,g_{k},g_{k+1}),\gamma_{(g_{k},g_{k+1})}\ket=\bra\overrightarrow{Z}_1(g_{k+1},g_{k+1},g_{k+2}), \gamma_{(g_{k+1},g_{k+2})}\ket
$$ if and only if
\begin{equation}\label{dynamiceq10}
\bra\overleftarrow{X}(g_k),\tilde{\mu}_{g_k}\ket-\bra\overrightarrow{X}(g_{k+1}),\bar{\mu}_{g_{k+1}}\ket=\bra\overrightarrow{X}(g_{k+1}),\mu_{g_{k+1}}\ket+\bra\overrightarrow{X}(g_{k+1}),\tilde{\mu}_{g_{k+1}}\ket.
\end{equation} for $k=1,\ldots,N-1$. Using \eqref{dynamiceq7}, \eqref{dynamiceq8} and
\eqref{dynamiceq9} equations \eqref{dynamiceq10} are equivalent to
\begin{align*}0=&\bra\overrightarrow{X}(g_{k+1}),D_1L_d(g_{k+1},g_{k+2})+D_2
L_d(g_k,g_{k+1})\ket\\&-\bra\overleftarrow{X}(g_k),D_1
L_d(g_k,g_{k+1})+D_2L_d(g_{k-1},g_{k})\ket\end{align*} i.e.,
\begin{align*}0=&\ell^{*}_{g_{k}}\left(D_1L_d(g_{k},g_{k+1})+D_2
L_d(g_{k-1},g_{k})\right)\\&+(r_{g_{k+1}}\circ
i)^{*}\left(D_1L_d(g_{k+1},g_{k+2})+D_2 L_d(g_k,g_{k+1})\right)\end{align*} for
$k=2,\ldots,N-2$, after a shifting of the indexes, as we claimed. \hfill$\square$



\begin{example} Let $G$ be a Lie group and let
$L_d:G_2=G\times G\rightarrow\mathbb{R}$ be a discrete second-order Lagrangian. The prolongation of $G$ over its source map, $\palg$, is a Lie groupoid over $G$ and it can be identified with three copies of $G$, that is, $\mathcal{P}^{\alpha}G\simeq 3G$. We construct a Lagrangian submanifold
of the cotangent groupoid $T^{*}(\palg)$ as \begin{equation}\label{sigmalg}\Sigma_{L_d}=\{\mu\in
T^{*}(\palg)\mid i^{*}_{\palg}\mu=dL_d\}\subseteq T^{*}(\palg)\end{equation}
where $i_{\palg}:G_2\to\palg$ is the inclusion given by $i_{\palg}(g,h)=(g,g,h)$ with $(g,h)\in G_2$.


Observe that  $\overleftarrow{X}(g_k)=T\ell_{g_k}(X),\; 
\overrightarrow{X}(g_k)=-T(r_{g_k}\circ i)(X)=Tr_{g_k}(X).$ Therefore, we have that a sequence 
  $\gamma_{(g_1,g_2)},\ldots,\gamma_{(g_{N-1},g_N)}$, where $\gamma_{(g_k,g_{k+1})}\in T^{*}_{i_{G_2}(g_k,g_{k+1})}(\palg),$ with $(g_k,g_{k+1})\in G_{2}$  for $k=1,\ldots,N-1$, satisfies the discrete second-order dynamics on $\Sigma_{L_d}$ if
\begin{align*}\label{dsoepeq}0=&\ell^{*}_{g_k}D_1L_d(g_k,g_{k+1})+\ell^{*}_{g_k}D_2L_d(g_{k-1},g_k)
-r_{g_{k+1}}^{*}D_2L_d(g_k,g_{k+1})\\&-r^{*}_{g_{k+1}}D_1L_d(g_{k+1},g_{k+2}),
\end{align*}  that is, $(g_k,g_{k+1})\in G_{2}$ for $k=1,\ldots,N-1$ is a solution of the discrete second-order Euler-Poincar\'e equations for the discrete second-order Lagrangian $L_{d}:G_{2}\to\R$.

\end{example}

\begin{example}
Let $Q$ be a differentiable manifold, consider the pair groupoid $Q\times Q\rightrightarrows Q$, where
the source and target maps are given by the projections onto the
fist and second factor, respectively. The set of admissible elements
is given by
$$(Q\times Q)_{2}=\{((q_{0},q_1),(\bar{q}_1,q_2))\in (Q\times Q)\times (Q\times Q)\mid q_1=\bar{q}_1\}\simeq 3Q.$$
The prolongation Lie groupoid is a Lie groupoid over $Q\times Q$ given by 
$$\mathcal{P}^{\alpha}(Q\times Q)=\{((q_0,q_1),(q_2,q_3),(q_4,q_5))\in 3(Q\times Q)\mid
q_1=q_2\hbox{ and }q_3=q_4\}\simeq 4Q,$$ where the inclusion
of $(Q\times Q)_{2}$ into $\mathcal{P}^{\alpha}(Q\times Q)$ is given by
\begin{eqnarray*}
i_{3Q}: (Q\times Q)_{2}\simeq 3Q&\hooklongrightarrow&\mathcal{P}^{\alpha}(Q\times Q)\simeq 4Q\\
(q_0,q_1,q_2)&\longmapsto&(q_0,q_1,q_1, q_2).
\end{eqnarray*} Given $L_d:(Q\times Q)_2\to\R$, a discrete second-order Lagrangian, we
construct the Lagrangian submanifold $\Sigma_{L_d}$ of the cotangent
groupoid $(T^{*}(\mathcal{P}^{\alpha}(Q\times Q)),\omega_{\mathcal{P}^{\alpha}(Q\times Q)})$ over $A^{*}(\mathcal{P}^{\alpha}(Q\times Q))$, where $\omega_{\mathcal{P}^{\alpha}(Q\times Q)}$ denotes the
canonical symplectic 2-form on $T^{*}(\mathcal{P}^{\alpha}(Q\times Q))$, by
$$\Sigma_{L_d}=\{\mu\in T^{*}(\mathcal{P}^{\alpha}(Q\times Q))\mid i_{3Q}^{*}\mu=dL_d\}.$$ Here $\mu=\mu_0 dq_0+\mu_1 dq_1+\bar{\mu}_1 d\bar{q}_1+\mu_2 dq_2$. Therefore, $\mu\in\Sigma_{L_d}$ if it satisfies
$$
\mu_0=\frac{\partial L_d}{\partial q_0}(q_0,q_1,q_2),\quad
\mu_1+\bar{\mu}_1=\frac{\partial L_d}{\partial q_1}(q_0,q_1,q_2),\quad
\mu_2=\frac{\partial L_d}{\partial q_2}(q_0,q_1,q_2).$$

Using the source and target map given by \begin{eqnarray*}
\tilde{\alpha}&:&T^{*}(\mathcal{P}^{\alpha}(Q\times Q))\longrightarrow T^{*}(Q\times Q)\quad
(\mu_0,\mu_1,\bar{\mu}_1,\mu_2)\longrightarrow
(-\mu_0,-\mu_1)\\
\tilde{\beta}&:&T^{*}(\mathcal{P}^{\alpha}(Q\times Q))\longrightarrow T^{*}(Q\times Q)\quad
(\mu_0,\mu_1,\bar{\mu}_1,\mu_2)\longrightarrow (\bar{\mu}_1,\mu_2);
\end{eqnarray*} we have that the second-order discrete dynamics on $\Sigma_{L_d}$ holds if and only if
$$D_2L_d(q_{1},q_{2},q_{3})+D_1L_d(q_2,q_{3},q_{4})+D_3L_d(q_{0},q_{1},q_{2})=0.$$

\end{example}

\subsection{Regularity conditions and Poisson structure}

We have seen how the dynamics is implicitly defined by a relation on
$T^{*}(\palg)$ rather than an explicitly defined map and pointed out that $\gamma_{(g_{1},g_{2})},\ldots,\gamma_{(g_{N-1},g_{N})}\in
T^{*}(\palg)$ satisfies the discrete second-order dynamics if and
only if for each pair of successive elements in $T^{*}(\palg)$ they
satisfy
\begin{equation}\label{relacionMMS}
(\gamma_{(g_{k},g_{k+1})},\gamma_{(g_{k+1},g_{k+2})})\in (T^{*}(\palg))_2\cap (\Sigma_{L_d}\times\Sigma_{L_d}),\quad k=1,\ldots,N-2.\end{equation}

Weinstein \cite{We} raised the question of how
regularity results for the pair groupoid $Q\times Q$ might be
generalized to arbitrary Lie groupoids $G\rightrightarrows Q$, and
this question was answered by Marrero et al. (see, Theorem
4.13 in \cite{MMM}). Here, we study an extension of this problem to discrete second-order systems
following the work of Marrero et al. \cite{MMS} for first order systems. The problem consists on finding under which conditions the 
relation \eqref{relacionMMS} is the graph of an explicit flow
$$\gamma_{(g_{k-1},g_{k})}\longmapsto\gamma_{(g_{k},g_{k+1})}$$
(at least locally) and what properties have such map.

Consider the source map of the cotangent groupoid
$T^{*}(\palg)$ restricted to the Lagrangian submanifold $\Sigma_{L_d}$, that is,
$\tilde{\alpha}\mid_{\Sigma_{L_d}}:\Sigma_{L_d}\to A^{*}(\palg).$ If this
map is a local diffeomorphism, then the Lagrangian flow is locally
given by $$\Gamma_{L_d}=(\tilde{\alpha}\mid_{\Sigma_{L_d}})^{-1}\circ
\tilde{\beta}\mid_{\Sigma_{L_d}}:\Sigma_{L_d}\longrightarrow\Sigma_{L_d}.$$

\begin{theorem}[Marrero, Mart\'in de Diego and Stern \cite{MMS}]\label{theoremequivalent}
Let $\Gamma$ a symplectic groupoid over a manifold $P$ with source and target maps $\alpha$ and $\beta$, respectively. Let $\Sigma$ be a Lagrangian submanifold $\Sigma\subset \Gamma$. Then the restricted source map  $\alpha\mid_{\Sigma}:\Sigma\to P$ is a local diffeomorphism if and only if the restricted target map $\beta\mid_{\Sigma}:\Sigma\to P$ is a local diffeomorphism.
\end{theorem}

A direct consequence of Theorem \eqref{theoremequivalent} is that given the symplectic groupoid $(T^{*}(\palg),\omega_{\palg})$, if 
$\Sigma_{L_d}\subset T^{*}(\palg)$ is the Lagrangian submanifold
generated by the second-order discrete Lagrangian
$L_d:G_2\rightarrow\R,$
$\tilde{\alpha}\mid_{\Sigma_{L_d}}:\Sigma_{L_d}\to A^{*}(\palg)$ is a local
diffeomorphism if and only if
$\tilde{\beta}\mid_{\Sigma_{L_d}}:\Sigma_{L_d}\to A^{*}(\palg)$ is a local
diffeomorphism. The applications $\tilde{\alpha}\mid_{\Sigma_{L_d}}$
and $\tilde{\beta}\mid_{\Sigma_{L_d}}$ play the role of
$\mathbb{F}L_{d}^{-}$ and $\mathbb{F}L_{d}^{+}$, respectively, in discrete mechanics (see Appendix B and \cite{CoFeMdD} for the higher-order case), that is, $$\mathbb{F}L_d^{+}=\tilde{\beta}\mid_{\Sigma_{L_d}}\hbox{ and } \mathbb{F}L_{d}^{-}=\tilde{\alpha}\mid_{\Sigma_{L_d}},$$ and therefore, from Theorem $\eqref{theoremequivalent}$, $\mathbb{F}L_{d}^{+}$ is a local diffeomorphism if and only if $\mathbb{F}L_{d}^{-}$ is a local
diffeomorphism.

Now, by Definition \eqref{defdynamics}, given a sequence
$\gamma_{(g_{1},g_{2})},\ldots,\gamma_{(g_{N-1},g_{N})}\in
T^{*}(\palg)$, it satisfies the
discrete second-order dynamics on $\Sigma_{L_d}$ if and only if
$\gamma_{(g_{1},g_{2})},\ldots,$ $\gamma_{(g_{N-1},g_{N})}\in \Sigma_{L_d}$
and
$$\mathbb{F}L_d^{+}(\gamma_{(g_{k},g_{k+1})})=\mathbb{F}L_d^{-}(\gamma_{(g_{k+1},g_{k+2})}),\quad k=1,\ldots,N-2.$$

\begin{definition}
Let $L_d:G_2\to\R$ be a discrete second-order Lagrangian. It is said to be \textit{regular} if $\tilde{\alpha}\mid_{\Sigma_{L_d}}$ is a local diffeomorphism, and \textit{hyperregular} if it is a global diffeomorphism. 
\end{definition}


\begin{definition}
If the Lagrangian $L_d:G_2\to\R$ is hyperregular, the map  defined as $\widetilde{\Gamma}_{L_d}=\tilde{\beta}\mid_{\Sigma_{L_d}}\circ(\tilde{\alpha}\mid_{\Sigma_{L_d}})^{-1}:A^{*}(\palg)\to
A^{*}(\palg)$ is called \textit{Hamiltonian evolution operator}. \end{definition}

The next theorem shows the relation between the Hamiltonian evolution operator and the preservation of the Poisson structure on $A^{*}(\palg)$.

\begin{theorem}\label{propositiongroupoids1} Assume that  $\tilde{\alpha}\mid_{\Sigma_{L_d}}$ is a  global diffeomorphisms. Then,  the discrete Hamiltonian evolution
operator $\widetilde{\Gamma}_{L_d}:A^{*}(\palg)\to A^{*}(\palg)$
preserves the Poisson structure on $A^{*}(\palg).$
\end{theorem}

\begin{proof}

If $\tilde{\alpha}\mid_{\Sigma_{L_d}}$ (or, equivalently,  $\tilde{\beta}\mid_{\Sigma_{L_d}}$) is a global diffeomorphism, the Hamiltonian evolution operatior $\tilde{\Gamma}_{L_d}$ is a global automorphism on $A^{*}(\palg)$. 

Consider the application 
\begin{eqnarray*}
(\tilde{\alpha},\tilde{\beta}):T^{*}(\palg)\longrightarrow&
\overline{A^{*}(\palg)}\times A^{*}(\palg)\\
\mu\longmapsto&(\tilde{\alpha}(\mu),\tilde{\beta}(\mu))
\end{eqnarray*} where
$\overline{A^{*}(\palg)}$ denotes $A^{*}(\palg)$ endowed with the
linear Poisson structure changed of sign. 

The submanifold $\Sigma_{L_d}\subset T^{*}(\palg)$ is the graph of  $\tilde{\Gamma}_{L_d}$ in $\overline{A^{*}(\palg)}\times A^{*}(\palg)$.
Since $(\tilde{\alpha},\tilde{\beta})$ is a Poisson map and $\Sigma_{L_d}$ is a Lagrangian submanifold, the image of $\Sigma_{L_d}$  by  $(\tilde{\alpha},\tilde{\beta})$ is a coisotropic submanifold of $\overline{A^{*}(\palg)}\times A^{*}(\palg)$.  Thus, by corollary (2.2.3) in \cite{We2},  $\tilde{\Gamma}_{L_d}$ is a (local) Poisson automorphism on $A^{*}(\palg)$ and therefore $\widetilde{\Gamma}_{L_d}:A^{*}(\palg)\to A^{*}(\palg)$
preserves the linear Poisson structure on $A^{*}(\palg)$ as we claimed.
\end{proof}

\subsection{Morphism, reduction and Noether symmetries}
In this subsection we study the reduction of discrete
second order Lagrangian systems and Noether symmetries. Consider two Lie groupoids $G\rightrightarrows Q$ and $G'\rightrightarrows Q'$ with
structural maps denoted by $\alpha$, $\beta$, $i$, $\epsilon$, $m$ and $\alpha'$, $\beta'$, $i'$, $\epsilon'$, $m'$ respectively.


\begin{definition}\label{morphism}
A smooth map $\chi: G\to G'$ is a \textit{morphism
of Lie groupoids} if, for every composable pair $(g, h)\in G_2,$ it
satisfies $(\chi(g), \chi(h))\in G'_2$ and $\chi(gh) =
\chi(g)\chi(h)$ where $G'_2$ denotes the set of composable pairs on $G'\rightrightarrows Q'$

\end{definition}

A morphism of Lie groupoids $\chi:G\to G'$ induces a smooth map
$\chi_0:Q\to Q'$ such that
$$\alpha'\circ\chi=\chi_0\circ\alpha,\quad \beta'\circ\chi=\chi_0\circ\beta,\hbox{ and }
\chi\circ\epsilon=\epsilon'\circ\chi_0,$$ that is, the following
diagram is commutative,
\[
\xymatrix{
G\ar[d]_{\chi} \ar[rr]^{\alpha, \beta} & \ & Q \ar[d]^{\chi_0} \\
G'\ar[rr]^{\alpha', \beta'} & \ & Q'
}
\]

A morphism of Lie groupoids $\chi$ induces a morphism $A\chi:AG\to AG'$ of their
corresponding Lie algebroids and
\begin{align}
\overleftarrow{A\chi(v)}(\chi(g))&=T_g\chi(\overleftarrow{v}(g)),\\
\overrightarrow{A\chi(w)}(\chi(g))&=T_g\chi(\overrightarrow{w}(g)),\label{ppp1}
\end{align} for all $g\in G,$ $v\in A_{\beta(g)}G$ and $w\in A_{\alpha(g)}G.$
Moreover, for $X\in\Gamma(AG)$ and ${X}'\in\Gamma(AG')$ we have $A\chi\circ X={X}'\circ\chi_0$ if and only if
$T\chi\circ\overleftarrow{X}=\overleftarrow{{X}'}\circ\chi$
(resp.,
$T\chi\circ\overrightarrow{X}=\overrightarrow{{X}'}\circ\chi$),  (see \cite{Mack} for more details). 
That is, $X$ and ${X}'$ are $``A\chi$-related'' if and only if
their corresponding left-invariant (reps., right invariant) vector
fields are $\chi$-related.

Now consider the prolongations of $G$ and $G'$ by the source map $\alpha$ and $\alpha'$, respectively, denoted by $\palg$ and $\palgt$.

\begin{definition}
Let $\chi:\palg\to\palgt$ be a morphism of Lie groupoids and $x\in\palg$. Two
covectors $\mu\in T^{*}_{x}(\palg)$ and $\mu'\in
T^{*}_{\chi(x)}(\palgt)$ are said to be $\chi^{*}$-related if
$\bra\mu,\xi\ket=\bra\mu',T\chi(\xi)\ket$ for all $\xi\in T_{x}(\palg).$
Also, if $z\in A^{*}_{x}(\palg)$ and $z'\in
A^{*}_{\chi_0(x)}(\palgt),$ are $A^{*}\chi-$related if $\bra
z,\widetilde{\xi}\ket=\bra z',A\chi(\widetilde{\xi})\ket$
for all $\widetilde{\xi}\in A_{x}(\palg)$, where
$\chi_0:\palg\to\palgt$ denotes the smooth map on the base induced
by the morphism $\chi$ and where $A\chi:A(\palg)\to A(\palgt)$ is the
associated Lie algebroid morphism.
\end{definition}


The following theorem states the reduction of
discrete second order Lagrangian systems on Lie groupoids. It follows Theorem $4.5$ in \cite{MMS} for discrete first order systems.

\begin{theorem}\label{morfismosdegroupoidesdelieth}

Consider two Lie groupoids $G\rightrightarrows Q$ and $G'\rightrightarrows Q'$. Let $L_d:G_{2}\to\R$ and $L_d':G_2'\to\R$ be a discrete second order Lagrangians, and let $\chi:\palg\to\palgt$ be a morphism of Lie groupoids satisfying $G_2=\chi^{-1}(G_2')$ and $L_d=L_d'\circ\chi\mid_{G_2}$.
If $\mu\in T^{*}(\palg)$ and $\mu'\in T^{*}(\mathcal{P}^{\alpha'}G')$ are
$\chi^{*}-$related, the following properties hold

\begin{enumerate}
\item If $\mu'\in\Sigma_{L'_d}$ then $\mu\in\Sigma_{L_d}$.
\item The sources $\tilde{\alpha}(\mu)\in A^{*}(\palg)$ and
$\tilde{\alpha}'(\mu')\in A^{*}(\mathcal{P}^{\alpha'}G')$ are
$A^{*}\chi$-related.
\item The targets $\tilde{\beta}(\mu)\in A^{*}(\palg)$ and
$\tilde{\beta'}(\mu')\in A^{*}(\mathcal{P}^{\alpha'}G')$ are
$A^{*}\chi$-related.
\end{enumerate}
\end{theorem}

\begin{proof}

\begin{itemize}
\item[(1)] Consider $\mu\in
T^{*}_x(\palg)$ and $v\in T_{x}(\palg)$ with $x\in\palg$. Since $\mu$ and
$\mu'$ are $\chi^{*}-$related we have
\begin{align*}
\bra \mu,v\ket&=\bra\mu',T\chi(v)\ket=\bra dL_d',T \chi\mid_{G_{2}}(v)\ket\\
&=\bra\chi^{*}(dL_d'),v\ket=\bra d(L_d'\circ\chi\mid_{G_{2}}),v\ket=\bra
dL_d,v\ket\end{align*} and therefore $\mu\in\Sigma_{L_d}$.

\item[(2)] Consider $v\in
A_{\alpha(x)}(\palg).$ Then
$\bra\tilde{\alpha}(\mu),v\ket=\bra\mu,\overrightarrow{v}\ket=\bra\mu',T\chi(\overrightarrow{v})\ket,$
because are $\chi^{*}-$related. Using \eqref{ppp1} we have 
$$\bra\mu',T\chi(\overrightarrow{v})\ket=\bra\mu',\overrightarrow{A\chi(v)}\ket=\bra\tilde{\alpha}'(\mu'),A\chi(v)\ket.$$
Therefore 
$\bra\tilde{\alpha}(\mu),v\ket=\bra\tilde{\alpha}'(\mu'),A\chi(v)\ket$,
and thus $\tilde{\alpha}(\mu)$ and $\tilde{\alpha}'(\mu')$ are
$A^{*}\chi$-related.

\item[(3)] Consider $v\in T_{x}(\palg)$ with $x\in\palg$ and observe that 
$$\bra\tilde{\beta}(\mu),v\ket=\bra\mu,\overleftarrow{v}\ket=\bra\mu',T\chi(\overleftarrow{v})\ket=\bra\mu',\overleftarrow{A\chi(v)}\ket=\bra\tilde{\beta'}(\mu'),A\chi(v)\ket.$$
Thus, $\tilde{\beta}(\mu)$ and $\tilde{\beta}'(\mu')$ are
$A^{*}\chi-$related.

\end{itemize}

 \end{proof}

\begin{corollary}\label{corolariogroupoids}
Let $\chi:\palg\to\palgt$ be a morphism of Lie groupoids. 

If $\gamma'_{(g'_1,g'_2)},\ldots,\gamma'_{(g'_{N-1},g'_{N})}\in
T^{*}(\palgt)$ satisfy the discrete second-order dynamics for the discrete system determined by
$L_d':G_2'\to\R$ then any sequence $\chi^{*}-$related
$\gamma_{(g_{1},g_{2})},\ldots,\gamma_{(g_{N-1},g_{N})}\in
T^{*}(\palg)$ satisfy the discrete second-order dynamics for the discrete system determined by $L_d:G_2\to\R$. \end{corollary}

\begin{proof} By Theorem \ref{morfismosdegroupoidesdelieth}, if $\mu'_k\in\Sigma_{L_{d}'}$ then
$\mu_k\in\Sigma_{L_d}$ for $k=1,\ldots,N$. Moreover, for all $v\in
A_{\beta(g_k,g_k,g_{k+1})}(\palg)=A_{\alpha(g_{k+1},g_{k+1},g_{k+2})}(\palg),$
using the fact  that $\mu$ and $\mu'$ are $\chi^{*}-$related, we have, for $\xi\in T^{*}(\palg)$,
\begin{align*}
\bra\tilde{\beta}(\gamma_{(g_{k},g_{k+1})}),\xi\ket&=\bra\tilde{\beta'}(\gamma_{(g'_{k},g'_{k+1})}'),A\chi(\xi)\ket\\
&=\bra\tilde{\alpha'}(\gamma_{(g'_{k+1},g'_{k+2})}'),A\chi(\xi)\ket\\
&=\bra\tilde{\alpha}(\gamma_{(g_{k+1}, g_{k+2})}),\xi\ket
\end{align*} and therefore $\tilde{\beta}(\gamma_{(g_{k},g_{k+1})})=\tilde{\alpha}(\gamma_{(g_{k+1},g_{k+2})})$ for $k=1,\ldots,N-2.$ That is, $\gamma_{(g_1,g_2)},$ $\ldots,$ $\gamma_{(g_{N-1},g_{N})}$ satisfy the discrete second-order dynamics for the discrete Lagrangian $L_d:G_2\to\R$.
\end{proof}

Finally, we introduce  the notion of Noether symmetry and
 constants of motion for discrete second-order Lagrangian systems and we prove that for all Noether symmetry of the discrete second-order Lagrangian system determined by $L_d:G_2\to\R$ there is a
corresponding constant of motion which is preserved by the discrete second-order dynamics. This is a natural extension of discrete Noether symmetry  for first order systems introduced in \cite{MMM} and \cite{MMS}.

\begin{definition}
A section $Z\in\Gamma(A(\palg))$ is said to be a Noether symmetry of
the discrete second-order Lagrangian system determined by $L_d:G_2\to\R$ if there exists
a function $f\in C^{\infty}(G)$ such that
$$\bra\tilde{\alpha}(\mu),Z(\alpha^{\alpha}(\xi))\ket+f(\alpha^{\alpha}(\xi))=\bra\tilde{\beta}(\mu),Z(\beta^{\alpha}(\xi))\ket+f(\beta^{\alpha}(\xi))
$$ for all $\mu\in\Sigma_L$, with
$\xi=\pi_{\palg}(\mu)$, where $\pi_{\palg}:T^{*}(\palg)\to\palg$
is the cotangent bundle projection and $\alpha^{\alpha},\beta^{\alpha}:\palg\to G$ are the source and target maps, respectively, of the Lie groupoid $\palg$ over $G$.
\end{definition}

When $f=0$, we say $L_d$ is \textit{invariant} with respect to $Z$, and the conserved quantity is $$F_{Z}(\mu)=\bra\mathbb{F}L_{d}^{\pm}(\mu),Z\ket,$$ and when $f\neq 0$, we say $L_d$ is \textit{quasi-invariant} with respect to $Z$.

\begin{theorem}\label{teoremanoethergroupoids}
If $Z\in\Gamma(A(\palg))$ is a Noether symmetry of a discrete
second-order Lagrangian system determined by the discrete second-order Lagrangian $L_d:G_2\to\R$, the function
$F_Z:\Sigma_L\to\R$ given by
$$F_Z(\mu)=\bra\tilde{\alpha}(\mu),Z(\alpha^{\alpha}(\xi))\ket+f(\alpha^{\alpha}(\xi))=\bra\tilde{\beta}(\mu),Z(\beta^{\alpha}(\xi))\ket+f(\beta^{\alpha}(\xi)),$$
is a constant of motion where $\xi=\pi_{\palg}(\mu)$. That is,
if $\gamma_{(g_1,g_2)},\ldots,\gamma_{(g_{N-1},g_N)}\in
T^{*}(\palg)$ satisfy the discrete second-order dynamics then,
$$F_Z(\gamma_{(g_k,g_{k+1})})=F_Z(\gamma_{(g_{k+1},g_{k+2})})$$ for
$k=1,\ldots,N-2.$
\end{theorem}

\begin{proof} If $\gamma_{(g_1,g_2)},\ldots,\gamma_{(g_{N-1},g_N)}$
satisfy the discrete second-order dynamics, then
$\tilde{\beta}(\gamma_{(g_k,g_{k+1})})=\tilde{\alpha}(\gamma_{(g_{k+1},g_{k+2})})$
where $\gamma_{(g_k,g_{k+1})}\in\Sigma_L$ for $k=1,\ldots,N-1.$
Therefore 

\begin{align*}
F_Z(\gamma_{(g_k,g_{k+1})})=&\bra\tilde{\beta}(\gamma_{(g_k,g_{k+1})}),Z(\beta^{\alpha}(g_k,g_k,g_{k+1}))\ket+f(\beta^{\alpha}(g_k,g_k,g_{k+1}))\\
=&\bra\tilde{\beta}(\gamma_{(g_k,g_{k+1})}),Z(g_{k+1})\ket+f(g_{k+1})\\
=&\bra\tilde{\alpha}(\gamma_{(g_{k+1},g_{k+2})}),Z(\alpha^{\alpha}(g_{k+1},g_{k+1},g_{k+2}))\ket\\ &+f(\alpha^{\alpha}(g_{k+1},g_{k+1},g_{k+2}))\\
=&F_Z(\gamma_{(g_{k+1},g_{k+2})})
\end{align*} for $k=1,\ldots,N-2.$
\end{proof}

\subsection{Discrete second-order constrained
mechanics}\label{discreteconstrained}

Consider a discrete second-order constrained systems, that is, given a Lie groupoid $G\rightrightarrows Q$ one consider the discrete second-order Lagrangian $L^N_d: N\to\R$ defined on a  submanifold $N$ of the set of composable elements $G_2$,. The submanifold $N$ implies  that the dynamics is restricted. 

From Theorem \ref{tulczyjew}
$$\Sigma_{L_{d}^N}=\{\mu\in T^{*}(\palg)\mid i_{N}^{*}\mu=dL_{d}^{N}\}$$ is a Lagrangian submanifold where $i_N:N\hookrightarrow \palg$ is the inclusion from $N$ to $\palg$.  

The Lagrangian submanifold, $\Sigma_{L_{d}^N}$ is  an affine bundle over $N$ taking the
projection $\pi_{\palg}\mid_{\Sigma_{L_{d}^N}}:\Sigma_{L_{d}^N}\to N$,
the restriction of the cotangent bundle projection $\pi_{\palg}:T^{*}(\palg)\to\palg$ to this Lagrangian submanifold.



Suppose that the constraint submanifold $N\subset G_{2}$ is given by
$$N=\{(g_{1},g_{2})\in G_{2}\mid \Phi^{A}(g_1,g_2)=0,\hbox{ with }A\in
I\},$$ where $\{\Phi^{A}\}_{A\in I}$ is a family of real functions
defined in a neighborhood of $N$ and $I$ is an index set. Then, an element $\mu\in\Sigma_{L_{d}^N}$ with
$(g_1,g_2)=\pi_{\palg}\mid_{\Sigma_{L_d}^{N}}(\mu),$ can be written as
\[ \mu=dL_d(g_1,g_2)+\lambda_{A}
d\Phi^{A}(g_1,g_2)=d(L_d+\lambda_{A}\Phi^{A})(g_1,g_2)\in\Sigma_{L_{d}^N}.
\] 
where $L_{d}: G_2\to\R$ is an arbitrary extension of  $L_{d}^{N}: N\to \R$.

In this sense, $\Sigma_{L^N_{d}}$ can be locally seen as the space consisting of the elements $(g_{1},g_2)\in N$ together the Lagrange multipliers $\lambda_{A}$ constraining $(g_{1},g_{2})$  to $N$.

Therefore, by Theorem \ref{teoremagroupoides1}, the
sequence  $(g_1,g_2,\ldots,g_N,\lambda^{1},\lambda^2,\ldots,\lambda^{N-1})$, is a solution of the
discrete second-order constrained Lagrangian system determined by $\widehat{L}_{d}=L_{d}+\lambda_A\Phi^{A}_d$ with
$(g_j,g_{j+1})\in N$ for $j=1,\ldots,N-1$ if it satisfies
\begin{align*}
0=&\Big{\langle}\overrightarrow{X}(g_{k+1}),D_1L_{d}(g_{k+1},g_{k+2})+\lambda^{k+1}_{A}D_{1}\Phi^{A}_d(g_{k+1},g_{k+2})+D_{2}L_{d}(g_{k},g_{k+1})\\
&+\lambda^{k}_{A}D_{2}\Phi^{A}_d(g_{k},g_{k+1})\Big{\rangle}-\Big{\langle}\overleftarrow{X}(g_{k}),D_1L_{d}(g_{k},g_{k+1})+\lambda^{k}_{A}D_{1}\Phi^{A}_d(g_{k},g_{k+1})\\
&+D_{2}L_{d}(g_{k-1},g_{k})+\lambda^{k-1}_{A}D_{2}\Phi^{A}_d(g_{k-1},g_{k})\Big{\rangle}, 
\end{align*}for $k=2,\ldots,N-2$ and $X$ a vector field on $G$.

Therefore, the sequence  $(g_1,g_2,\ldots,g_N,\lambda^{1},\lambda^2,\ldots,\lambda^{N-1})$ 
satisfies \begin{align*}
0=&\Phi^{A}_d(g_{k},g_{k+1}),\hbox{ for all } A\in I,\quad k=1,\ldots, N-1;\\
0=&\ell_{g_{k}}^{*}\Big(D_1L_{d}(g_{k},g_{k+1})+\lambda^{k}_{A}D_{1}\Phi^{A}_d(g_{k},g_{k+1})+D_{2}L_{d}(g_{k-1},g_{k})\\
&+\lambda^{k-1}_{A}D_{2}\Phi^{A}_d(g_{k-1},g_{k})\Big)
+(r_{g_{k+1}}\circ i)^{*}\Big(D_1L_{d}(g_{k+1},g_{k+2})\\&+\lambda^{k+1}_{A}D_{1}\Phi^{A}_d(g_{k+1},g_{k+2})+D_{2}L_{d}(g_{k},g_{k+1})+\lambda^{k}_{A}D_{2}\Phi^{A}_d(g_{k},g_{k+1})\Big),
\end{align*} for $k=2,\ldots,N-2$.

\begin{remark}
When the Lie groupoid is a Lie group, we obtain the
second-order Euler-Poincar\'e equations for systems with constraints
(see \cite{CojiMdD} for example)
\begin{align*}
0=&\Phi^{A}_d(g_{k},g_{k+1}),\quad 0=\Phi^{A}_d(g_{k-1},g_{k}),\quad 0=\Phi^{A}_d(g_{k+1},g_{k+2})\hbox{ for all }A\in I;\\
0=&\ell_{g_{k}}^{*}\Big(D_1L_{d}(g_{k},g_{k+1})+\lambda^{k}_{A}D_{1}\Phi^{A}_d(g_{k},g_{k+1})+D_{2}L_{d}(g_{k-1},g_{k})\\
&+\lambda^{k-1}_{A}D_{2}\Phi^{A}_d(g_{k-1},g_{k})\Big)
-r_{g_{k+1}}^{*}\Big(D_1L_{d}(g_{k+1},g_{k+2})+\lambda^{k+1}_{A}D_{1}\Phi^{A}_d(g_{k+1},g_{k+2})\\
&+D_{2}L_{d}(g_{k},g_{k+1})+\lambda^{k}_{A}D_{2}\Phi^{A}_d(g_{k},g_{k+1})\Big),
\end{align*} for  $k=2,\ldots,N-2$.

When the Lie groupoid is the Banal groupoid, we have
\begin{align*}
0=&\Phi^{A}_d(q_{k},q_{k+1},q_{k+2}),\quad 0=\Phi^{A}_d(q_{k-1}, q_{k},q_{k+1}),\quad 0=\Phi^{A}_d(q_{k-2},q_{k-1},q_{k})\\
0=&D_1L_{d}(q_{k},q_{k+1},q_{k+2})+\lambda^{k}_{A}D_{1}\Phi^{A}_d(q_{k},q_{k+1},q_{k+2})+D_{2}L_{d}(q_{k-1},q_{k},q_{k+1})\\
&+\lambda^{k-1}_{A}D_{2}\Phi^{A}_d(q_{k-1},q_{k},q_{k+1})+D_3L_{d}(q_{k-2},q_{k-1},q_k)\\
&+\lambda^{k-2}_{A}D_{3}\Phi^{A}_d(q_{k-2},q_{k-1},q_k),
\end{align*} for all $A\in I$ and for $k=2,\ldots,N-2$. The equations given above are the discrete second-order Euler-Lagrange
equations for systems with second-order constraints (see
\cite{ldm} for example).\hfill$\diamond$

\end{remark}

\section{Conclusions and future research}

In this paper, we  have developed  a generalized theory
of discrete second-order Lagrangian mechanics from a variational point of view and we have  shown how to apply this theory to the construction of
variational integrators for some interesting examples of optimal control problems of mechanical
systems. After that, we have  shown how  Lagrangian submanifolds of a
symplectic groupoid (cotangent groupoid) give rise an intrinsic way  to discrete
dynamical second-order systems, and we have studied the geometric properties of these
systems from the
perspective of symplectic and Poisson geometry. Finally,  we have developed the reduction by Noether symmetries, and we have studied the relationship between the dynamics and variational principles for these second-order
variational problems. 

In \cite{BlCoGuMdD} we have studied optimal control problems for nonholonomic mechanical systems as second-order constrained variational problems. Let $\mathcal{D}$ be a non-integrable distribution defined by the nonholonomic constraints of some mechanical systems.  
We define the submanifold $\mathcal{D}^{(2)}$ of $T\mathcal{D}$ by
\begin{equation}\label{D2}
\mathcal{D}^{(2)}:=\{v\in T\mathcal{D}\mid v=\dot{\gamma}(0)\hbox{
where } \gamma:I\rightarrow\mathcal{D} \hbox{ is admissible}\},
\end{equation} and where we choose coordinates $(x^{i},y^{\alpha},\dot{y}^{\alpha})$ on
$\mathcal{D}^{(2)}$, and where the inclusion on $T\mathcal{D}$,
$i_{\mathcal{D}^{(2)}}:\mathcal{D}^{(2)}\hookrightarrow
T\mathcal{D}$ is given by
$$i_{\mathcal{D}^{(2)}}(q^{i},y^{\alpha},\dot{y}^{\alpha})=(q^{i},y^{\alpha},\rho_{\alpha}^{i}(q)y^{\alpha},\dot{y}^{\alpha}).$$ Therefore, $\mathcal{D}^{(2)}$ is locally described by the constraints on $T\mathcal{D}$ $$\dot{q}^{i}-\rho_{\alpha}^{i}y^{\alpha}=0.$$

The optimal control problem is determined by a smooth function
$\widetilde{L}:\mathcal{D}^{(2)}\rightarrow\mathbb{R}$ given a cost functional as in Section 3. To derive the equations of motion for $\widetilde{L}$ one can use
standard variational calculus for systems with constraints defining
the extended Lagrangian $\mathcal{L}$,
$$\mathcal{L}=\widetilde{L}+\lambda_{i}(\dot{q}^{i}-\rho_{\alpha}^{i}y^{\alpha}).$$

In a future work we would like to build variational integrators as an
alternative way to construct integration schemes for the type  of
optimal control problems studied in \cite{BlCoGuMdD}. Since the space
$\mathcal{D}^{(2)}$ is a subset of $T\mathcal{D}$ we can discretize
the tangent bundle $T\mathcal{D}$ by the cartesian product
$\mathcal{D}\times\mathcal{D}$. Therefore, our discrete variational
approach for optimal control problems of nonholonomic mechanical
systems will be determined by the construction of a discrete
Lagrangian $\widetilde{L}_{d}:\mathcal{D}_{d}^{(2)}\to\R$ where
$\mathcal{D}^{(2)}_{d}$ is the subset of
$\mathcal{D}\times\mathcal{D}$ locally determined by imposing the
discretization of the constraint $\dot{q}^{i}=\rho_{\alpha}^{i}(q)y^{\alpha}$. 
For instance we can consider
$$\mathcal{D}_{d}^{(2)}=\left\{(q_0^{i},y_{0}^{\alpha},q_{1}^{i},y^{\alpha}_{1})\in
\mathcal{D}\times\mathcal{D}\ \bigg{|}\ \frac{q_{1}^{i}-q_0^{i}}{h}=\rho_{\alpha}^{i}\left(\frac{q_0^{i}+q_1^{i}}{2}\right)\left(\frac{y_{0}^{\alpha}+y_{1}^{\alpha}}{2}\right)\right\}.$$

Now the system is in a form appropriate for the application of discrete
variational methods for constrained systems developed in this work from both, variational and geometrical points of view as in sections \ref{constrainedcase} and \ref{discreteconstrained}.

\section*{Appendix A: Higher-order tangent bundles}
In this Appendix we recall some basic facts of the geometry of
tangent bundle theory.  For more details see \cite{CSC,LR1}.

Let $Q$ be a  differentiable manifold of dimension $n$. It is possible to introduce an equivalence relation
 in the set $C^{k}(\R, Q)$ of $k$-differentiable
curves from $\R$ to $Q$. By definition,  two given curves in $Q$,
$\gamma_1(t)$ and $\gamma_2(t)$,
where $t\in (-a, a)$ with $a\in \R$
have contact of order  $k$ at $q_0 = \gamma_1(0) = \gamma_2(0)$ if
there is a local chart $(\varphi, U)$ of $Q$ such that $q_0 \in U$
and
$$\frac{d^s}{dt^s}\left(\varphi \circ \gamma_1(t)\right){\big{|}}_{t=0} =
\frac{d^s}{dt^s} \left(\varphi
\circ\gamma_2(t)\right){\Big{|}}_{t=0}\; ,$$ for all $s = 0,...,k.$ This
is a well defined equivalence relation
in $C^{k}(\R,Q)$
and the equivalence class of a  curve $\gamma$ will be denoted by
$[\gamma ]_0^{(k)}.$ The set of equivalence classes will be
denoted by $T^{(k)}Q$
and it is not hard to show that it has a natural structure of differentiable manifold. Moreover, $ \tau_Q^k  : T^{(k)} Q
\rightarrow Q$ where $\tau_Q^k \left([\gamma]_0^{(k)}\right) =
\gamma(0)$ is a fiber bundle called the \emph{tangent bundle of
order $k$} of $Q.$

From a local chart $q^{(0)}=(q^i)$ on a neighborhood $U$ of $Q$ with
$i=1,\ldots,n$, it is possible to induce local coordinates
 $(q^{(0)},q^{(1)},\dots,q^{(k)})$ on
$T^{(k)}U=(\tau_Q^k)^{-1}(U)$. The standard
convention is, $q^{(0)}\equiv q^{i}$, $q^{(1)}\equiv \dot{q}^i$ and
$q^{(2)}\equiv \ddot{q}^i$.




\section*{Appendix B: Discrete Mechanics}

This appendix briefly reviews some key results of discrete mechanics
(see Marsden and West \cite{MaWest} for more details).
\subsection*{A.1. Discrete Lagrangian Mechanics}

A \textit{discrete Lagrangian} is a differentiable function
$\mathbb{L}_d\colon Q \times Q\to \R$, which may be considered as an
approximation of the action integral defined by a continuous regular 
Lagrangian $L\colon TQ\to \R.$ That is, given a time step $h>0$
small enough,
\[
\mathbb{L}_d(q_0, q_1)\approx \int^h_0 L(q(t), \dot{q}(t))\; dt,
\]
where $q(t)$ is the unique solution of the Euler-Lagrange equations
for $L$ with  boundary conditions $q(0)=q_0$ and $q(h)=q_1$.

We construct the grid $\{t_{k}=kh\mid k=0,\ldots,N\},$ with $Nh=T$
and define the discrete path space
$\mathcal{P}_{d}(Q):=\{q_{d}:\{t_{k}\}_{k=0}^{N}\to Q\}.$ We
identify a discrete trajectory $q_{d}\in\mathcal{P}_{d}(Q)$ with its
image $q_{d}=\{q_{k}\}_{k=0}^{N}$ where $q_{k}:=q_{d}(t_{k}).$ The
discrete action $\mathcal{A}_{d}:\mathcal{P}_{d}(Q)\to\R$ along this
sequence is calculated by summing the discrete Lagrangian on each
adjacent pair and defined by \begin{equation}\label{acciondiscreta}
\mathcal{A}_d(q_{d}) = \mathcal{A}_d(q_0,...,q_N) :=\sum_{k=0}^{N-1}\mathbb{L}_d(q_k,q_{k+1}).
\end{equation}
We would like to point out that the discrete path space is
isomorphic to the smooth product manifold which consists on $N+1$
copies of $Q,$ the discrete action inherits the smoothness of the
discrete Lagrangian and the tangent space
$T_{q_{d}}\mathcal{P}_{d}(Q)$ at $q_{d}$ is the set of maps
$v_{q_{d}}:\{t_{k}\}_{k=0}^{N}\to TQ$ such that $\tau_{Q}\circ
v_{q_{d}}=q_{d}$ which will be denoted by
$v_{q_{d}}=\{(q_{k},v_{k})\}_{k=0}^{N},$ where $\tau_{Q} : TQ
\rightarrow Q$ is the canonical projection.

For any product manifold $Q_1\times Q_2,$
$T^{*}_{(q_1,q_2)}(Q_1\times Q_2)\simeq T^{*}_{q_1}Q_1\times
T^{*}_{q_2}Q_2,$ for $q_1\in Q_1$ and $q_2\in Q_2$ where $T^{*}Q$
denotes the cotangent bundle of a differentiable manifold $Q.$
Therefore, any covector $\alpha\in T^{*}_{(q_1,q_2)}(Q_1\times Q_2)$
admits an unique decomposition $\alpha=\alpha_1+\alpha_2$ where
$\alpha_i\in T^{*}_{q_i}Q_i,$ for $i=1,2.$ Thus, given a discrete
Lagrangian $\mathbb{L}_d$ we have the following decomposition
  $$\hbox{d}\mathbb{L}_{d}(q_0,q_1)=D_{1}\mathbb{L}_d(q_0,q_1)+D_{2}\mathbb{L}_d(q_0,q_1),$$
where $D_{1}\mathbb{L}_d(q_0,q_1)\in T^*_{q_0}Q$ and $D_{2}\mathbb{L}_d(q_0,q_1)\in
T^*_{q_1}Q$.

The discrete variational principle, or Cadzow's principle \cite{Cad}, states that
the solutions of the discrete system determined by $\mathbb{L}_d$ must
extremize the action sum given fixed points $q_0$ and $q_N.$
Extremizing $\mathcal{A}_d$ over $q_k$ with $1\leq k\leq N-1,$ we
obtain the following system of difference equations
%
\begin{equation}\label{discreteeq}
 D_1\mathbb{L}_d( q_k, q_{k+1})+D_2\mathbb{L}_d( q_{k-1}, q_{k})=0.
\end{equation}
These equations are usually called \textit{discrete Euler-Lagrange
equations}. Given a solution $\{q_{k}^{*}\}_{k\in\mathbb{Z}}$ of
eq.\eqref{discreteeq} and assuming the regularity hypothesis (the
matrix $(D_{12}\mathbb{L}_d(q_k, q_{k+1}))$ is regular), it is possible to
define implicitly a (local) discrete flow $
\Upsilon_{\mathbb{L}_d}\colon\mathcal{U}_{k}\subset Q\times Q\to Q\times Q$,
by $\Upsilon_{\mathbb{L}_d}(q_{k-1}, q_k)=(q_k, q_{k+1})$ from
(\ref{discreteeq}) where $\mathcal{U}_{k}$ is a neighborhood of the
point $(q_{k-1}^{*},q_{k}^{*})$.




Let us define the discrete Lagrangian 1-forms $\Theta_{\mathbb{L}_{\rm
d}}^{\pm}: Q \times Q \to T^{*}(Q \times Q)$ by
\begin{subequations}
  \label{eq:DLSOF}
  \begin{align}
    \Theta_{\mathbb{L}_{\rm d}}^{+}&: (q_{k}, q_{k+1}) \longmapsto D_{2}\mathbb{L}_{\rm d}(q_{k}, q_{k+1})\,dq_{k+1},
    \\
    \Theta_{\mathbb{L}_{\rm d}}^{-}&: (q_{k}, q_{k+1}) \longmapsto -D_{1}\mathbb{L}_{\rm d}(q_{k}, q_{k+1})\,dq_{k}.
  \end{align}
\end{subequations}
Then, the discrete flow $\Upsilon_{\mathbb{L}_{\rm d}}$ preserves the discrete
Lagrangian form
\begin{equation}
  \label{eq:DLSTF}
  \Omega_{\mathbb{L}_{\rm d}}(q_{k}, q_{k+1}) =-d\Theta_{\mathbb{L}_{\rm d}}^{+} =-d\Theta_{\mathbb{L}_{\rm d}}^{-} =D_{1}D_{2}\mathbb{L}_{\rm d}(q_{k}, q_{k+1})\,dq_{k} \wedge dq_{k+1}.
\end{equation}
Specifically, we have
\begin{equation*}
  (\Upsilon_{\mathbb{L}_{\rm d}})^{*} \Omega_{\mathbb{L}_{\rm d}} = \Omega_{\mathbb{L}_{\rm d}}.
\end{equation*}

\subsection*{B.2. Discrete Hamiltonian Mechanics}

Introduce the {\em right and left discrete Legendre transformations}
$\mathbb{F}\mathbb{L}_{\rm d}^{\pm}: Q \times Q \to T^{*}Q$ by
\begin{subequations}
  \label{eq:DLT}
  \begin{align}
    \label{eq:DLT+}
    \mathbb{F}\mathbb{L}_{\rm d}^{+}&: (q_{k}, q_{k+1}) \longmapsto (q_{k+1}, D_{2}\mathbb{L}_{\rm d}(q_{k}, q_{k+1}) ),
    \\
    \label{eq:DLT-}
    \mathbb{F}\mathbb{L}_{\rm d}^{-}&: (q_{k}, q_{k+1}) \longmapsto (q_{k}, -D_{1}\mathbb{L}_{\rm d}(q_{k}, q_{k+1}) ),
  \end{align}
\end{subequations}
respectively. Then we find that the Eq.~\eqref{eq:DLSOF} and
\eqref{eq:DLSTF} are pull-backs by these maps of the Liouville
1-form $\lambda_Q$ and the canonical symplectic 2-form $\omega_Q$, on $T^{*}Q$, respectively, as follows:
\begin{equation*}
  \Theta_{\mathbb{L}_{\rm d}}^{\pm} = (\mathbb{F}\mathbb{L}_{\rm d}^{\pm})^{*} \lambda_{Q},
  \qquad
  \Omega_{\mathbb{L}_{\rm d}}^{\pm} = (\mathbb{F}\mathbb{L}_{\rm d}^{\pm})^{*} \omega_{Q}.
\end{equation*}
Let us define the momenta
\begin{equation*}
  p_{k,k+1}^{-} = -D_{1}\mathbb{L}_{\rm d}(q_{k}, q_{k+1}),
  \qquad
  p_{k,k+1}^{+} = D_{2}\mathbb{L}_{\rm d}(q_{k}, q_{k+1}).
\end{equation*}
Then, the discrete Euler--Lagrange equations become
simply $p_{k-1,k}^{+} = p_{k,k+1}^{-}$. So defining
\begin{equation*}
  p_{k} = p_{k-1,k}^{+} = p_{k,k+1}^{-},
\end{equation*}
one can rewrite the discrete Euler--Lagrange equations
as follows:
\begin{equation}
  \label{eq:DEL-H}
  \begin{array}{l}
     p_{k} = -D_{1}\mathbb{L}_{\rm d}(q_{k}, q_{k+1}),
    \medskip\\
     p_{k+1} = D_{2}\mathbb{L}_{\rm d}(q_{k}, q_{k+1}).
  \end{array}
\end{equation}
Furthermore, define the {\em discrete Hamiltonian map}
$\tilde{F}_{\mathbb{L}_{\rm d}}: T^{*}Q \to T^{*}Q$ by
\begin{equation}
  \label{eq:DHM}
  \tilde{F}_{\mathbb{L}_{\rm d}}: (q_{k}, p_{k}) \mapsto (q_{k+1}, p_{k+1}).
\end{equation}
Then, one may relate this map with the discrete Legendre transforms
in Eq.~\eqref{eq:DLT} as follows:
\begin{equation}
  \label{eq:DHM-DLT}
  \tilde{F}_{\mathbb{L}_{\rm d}} = \mathbb{F}\mathbb{L}_{\rm d}^{+} \circ (\mathbb{F}\mathbb{L}_{\rm d}^{-})^{-1}.
\end{equation}
Furthermore, one can also show that this map is symplectic, i.e.,
\begin{equation*}
  (\tilde{F}_{\mathbb{L}_{\rm d}})^{*} \omega_Q = \omega_Q.
\end{equation*}
This corresponds to the Hamiltonian description of the dynamics defined by the
discrete Euler--Lagrange equations introduced by
Marsden and West in \cite{MaWest}. Notice, however, that no discrete
analogue of Hamilton's equations is introduced here, although the
flow is now on the cotangent bundle $T^{*}Q$.

\section*{Appendix C: Prolongation of Lie algebroids and Mechanics on Lie algebroids}\label{ApC}

In this Appendix we recall the definition of the prolongation of a Lie agebroid $\tau_A:A\to Q$ over its projection map and the Euler-Lagrange equations on Lie algebroids. Further details can be found in \cite{LMM}, \cite{Eduardo} and \cite{Eduardoalg}.

\subsection*{C.1. Prolongation of a Lie algebroid}

Let  $(A,\llbracket\cdot,\cdot\rrbracket,\rho)$ be a Lie algebroid of rank $n$ over
$Q$ with projection $\tau_{A}:A\to Q$.

The prolongation of $A$ over its canonical projection, also called \textit{$A$ tangent bundle to $A$}, is defined to be 
$$\mathcal{T}^{\tau_A}A=\bigcup_{a\in A}\{(a',v_a)\in A\times
T_{a}A\mid\rho(a')=(T_{a}\tau_A)(v_a)\}$$ where
$T\tau_A:TA\rightarrow TQ$ is the tangent map to $\tau_A$.

In fact, $\mathcal{T}^{\tau_A}A$ is a Lie algebroid of rank $2n$
over $A$  where $\tau_{A}^{(1)}:\mathcal{T}^{\tau_A}A\Flder A$ is
the vector bundle projection given by 
$\tau_{A}^{(1)}(a',v_{a})=\tau_{TA}(v_{a})=a,$ and the anchor map is
$\rho_1:=pr_2:\mathcal{T}^{\tau_A}A\rightarrow TA$, the projection
over the second factor
(see \cite{LMM} and \cite{Eduardoalg} for more details).

If we now denote by $(a,a',v_a)$ an element
$(a',v_a)\in\mathcal{T}^{\tau_{A}}A$ where $a\in A$ and where $v$ is
tangent, we rewrite the definition of the prolongation of the Lie
algebroid as the subset of $A\times A\times TA$ given by $$\mathcal{T}^{\tau_A}A= \{(a,a',v_a)\in A\times A\times TA\mid \rho(a')=(T\tau_A)(v_a),
v_a\in T_{a}A
 \hbox{ and }\tau_A(a)=\tau_A(a')\}.$$ In this sense, if $(a,a',v_a)\in \mathcal{T}^{\tau_A}A$, then $\rho_1(a,a',v_{a})=(a,v_a)\in T_{a}A$, and the projection is given by $\tau_A^{(1)}(a,a',v_a)=a\in A$.

The prolongation of $A$ over $\tau_A$ takes the role of $TTQ$ in standard Lagrangian mechanics.

Along the paper when $A=AG\rightarrow Q$ is a Lie algebroid associated to a Lie groupoid we have that
$\mathcal{T}^{\tau_{AG}}AG=A({\mathcal P}^{\alpha}G)$

\subsubsection*{Example} Let $\al$ be a  finite dimensional real Lie algebra.  $\al$ is a
Lie algebroid over a single point $Q=\{q\}.$ Using that the anchor
map of $\al$ is zero we obtain that
$$\mathcal{T}^{\tau_{\al}}\al=\{(\xi_1,\xi_2,v_{\xi_1})\in\al\times T\al\}\simeq \al\times\al\times\al\simeq 3\al.$$ The anchor map of $\mathcal{T}^{\tau_{\al}}\al$ is $\rho_{1}(\xi_1,\xi_2,\xi_3)=(\xi_1,\xi_3)\in T_{\xi_1}\al$
and the Lie bracket is defined by $\llbracket
(\xi,\tilde{\xi}),(\eta,\tilde{\eta})\rrbracket=([\xi,\eta],0).$

\subsection*{C.2. Mechanics on Lie algebroids} (see \cite{Eduardo})

Let $A$ be a Lie algeroid over $Q$. We take local coordinates $(q^i)$ on $Q$ and a local basis $\{e_{\alpha}\}$ of sections of the vector bundle
$\tau_{A}:A\to Q$ with $\alpha=1,\ldots,n,$ then we have the
corresponding local coordinates on an open subset $\tau_{A}^{-1}(U)$
of $A$, $(q^i,y^{\alpha})$ ($U$ is an open subset of $Q$), where $y^{\alpha}(a)$
is the $\alpha$-th coordinate of $a\in A$ in the given basis i.e., every
$a\in A$ is expressed as
$a=y^{1}e_{1}(\tau_{A}(a))+\ldots+y^{n}e_{n}(\tau_{A}(a))$.

Such coordinates determine local functions $\rho_{\alpha}^i$,
$\mathcal{C}_{\alpha\beta}^{\gamma}$ on $Q$ which contain the local information of the Lie algebroid structure, and accordingly they are called {\it
structure functions of the Lie algebroid.} They are given by
\begin{equation}\label{estruct0}
\rho(e_{\alpha})=\rho_{\alpha}^i\frac{\partial }{\partial
q^i}\;\;\;\mbox{ and }\;\;\; \lcf
e_{\alpha},e_{\beta}\rcf=\mathcal{C}_{\alpha\beta}^{\gamma} e_{\gamma}.
\end{equation}

These functions should satisfy the relations
\begin{equation}\label{estruc1}
\rho_{\alpha}^j\frac{\partial \rho_{\beta}^i}{\partial q^j}
-\rho_{\beta}^j\frac{\partial \rho_{\alpha}^i}{\partial q^j}=
\rho_{\gamma}^i\mathcal{C}_{\alpha\beta}^{\gamma} \end{equation}

and

\begin{equation}\label{estruc2}
\sum_{cyclic(\alpha,\beta,\gamma)}[\rho_{\alpha}^i\frac{\partial
\mathcal{C}_{\beta\gamma}^\nu}{\partial q^i} + \mathcal{C}_{\alpha\mu}^{\nu}
\mathcal{C}_{\beta\gamma}^{\mu}]=0,
\end{equation}
which are usually called {\it the structure equations.}

Given a Lagrangian $L:A\to\R$, we fix two points $q_0$, $q_T$ in
the base manifold $Q$, then we look for
\emph{admissible curves} $\xi:I\subset\R\to A$, (i.e., curves on $A$ such that 
$\rho(\xi(t))=\frac{d}{dt}\tau_{A}(\xi(t))$) satisfying the variational principle  
$$0=\delta\int_{0}^{T}L(\xi(t))dt.$$

The infinitesimal variations are
$\delta \xi=\eta^{C},$ for all time-dependent sections
$\eta\in\Gamma(\tau_{A}),$ with $\eta(0)=0$ and $\eta(T)=0;$ where
$\eta^{C}$ is a time-dependent vector field on $A,$ the
\textit{complete lift}, locally defined by
$$\eta^C=\rho_{\alpha}^i\eta^{\alpha}\frac{\partial}{\partial q^i}+
(\dot{\eta}+\mathcal{C}^{\alpha}_{\beta\gamma}\eta^{\beta}y^{\gamma})\frac{\partial}{\partial
y^{\alpha}}.$$

From this variational principle (see \cite{Eduardo} for more details) one can derive the \textit{Euler-Lagrange
equations on Lie algebroids} $$ \frac{d}{dt}\left(\frac{\partial L}{\partial
y^{\alpha}}\right)-\rho_{\alpha}^{i}\frac{\partial L}{\partial
q^{i}}+\mathcal{C}_{\alpha\beta}^{\gamma}(q)y^{\beta}\frac{\partial
L}{\partial y^{\gamma}}=0,\quad
\frac{dq^i}{dt}=\rho_{\alpha}^{i}y^{\alpha}.$$

\section*{Appendix D: The Cayley map}\label{ApD}

The Cayley map $\ca:\al\Flder G$ is defined by
\[
\ca(\xi)=\left(e-\frac{\xi}{2}\right)^{-1}\left(e+\frac{\xi}{2}\right)
\]
where $e$ is the identity element of $G$. The Cayley map is valid for a class of quadratic groups (see \cite{Hair} for
example) that include the most interesting Lie groups in mechanics and the one studied in this paper, 
$SO(3)$. Its right trivialized derivative and
inverse are defined by
$$
\mbox{d}\ca_{x}\,y=\left(e-\frac{x}{2}\right)^{-1}\,y\,\left(e+\frac{x}{2}\right)^{-1},\quad
\mbox{d}\ca_{x}^{-1}\,y=\left(e-\frac{x}{2}\right)\,y\,\left(e+\frac{x}{2}\right).
$$

\subsection*{D.1. The Cayley map for $SO(3)$} The group of rigid body
rotations is represented by $3\times 3$ matrices with orthonormal
column vectors corresponding to the axes of a right-handed frame
attached to the body. On the other hand, the algebra $\alg$ is the
set of $3\times 3$ antisymmetric matrices. A $\alg$-basis can be
constructed as $\lc\hat e_{1},\hat e_{2},\hat e_{3}\rc$, $\hat
e_{i}\in\alg$ using the hat map $\hat{\cdot}:\R^{3}\to\mathfrak{so}(3)$, where $\lc e_{1},e_{2},e_{3}\rc$ is the standard
basis for $\R^{3}$ (see \cite{holmbook} for example). Elements $\xi\in\alg$ can be identified with the
vector $\omega\in\R^{3}$ through $\xi=\omega^{\alpha}\,\hat
e_{\alpha}$, or $\xi=\hat\omega$. Under such identification the Lie
bracket coincides with the standard cross product, i.e.,
$\ad_{\hat\omega}\,\hat\rho=\omega\times\rho$, for some
$\rho\in\R^{3}$. Using this identification we have
\begin{equation}\label{caySO}
\ca(\hat\omega)=I_{3}+\frac{4}{4+\parallel\omega\parallel^{2}}\lp\hat\omega+\frac{\hat\omega^{2}}{2}\rp,
\end{equation}
where $I_{3}$ is the $3\times 3$ identity matrix. The right trivialized derivative and
inverse are expressed as the $3\times 3$ matrices
\begin{equation}\label{Dtau}
\mbox{d}\ca_{\omega}=\frac{2}{4+\parallel\omega\parallel^{2}}(2I_{3}+\hat\omega),\,\,\,\,\,\mbox{d}\ca_{\omega}^{-1}=I_{3}-\frac{\hat\omega}{2}+\frac{\omega\,\omega^{T}}{4}.
\end{equation}

\section*{Acknowledgments} 
This work has been supported by MICINN (Spain) Grant MTM 2013-42870-P, ICMAT Severo Ochoa Project
SEV-2011-0087 and IRSES-project ``Geomech-246981''. We would like to thank Juan Carlos Marrero and Eduardo Mart\'inez for fruitful comments and discussions.

\medskip
Received xxxx 20xx; revised xxxx 20xx.
\medskip

\end{document}